\documentclass[reqno]{amsart}
\usepackage{amssymb}
\usepackage{amsmath}
\usepackage{amsthm}
\usepackage{mathtools}
\usepackage{graphicx}
\usepackage{comment}

\usepackage{hyperref}
\hypersetup{colorlinks}

\usepackage{amssymb, mathrsfs, amsfonts, amsmath}

\usepackage{tikz}
\usetikzlibrary{arrows}

\usepackage[english]{babel}

\usepackage[margin=1.4in, centering]{geometry}

\DeclareSymbolFont{bbold}{U}{bbold}{m}{n}
\DeclareSymbolFontAlphabet{\mathbbold}{bbold}
\newtheorem{thm}{Theorem}
\newtheorem{prop}[thm]{Proposition}
\newtheorem{lem}[thm]{Lemma}
\newtheorem{cor}[thm]{Corollary}
\theoremstyle{definition}
\newtheorem{defn}[thm]{Definition}
\newtheorem{ex}[thm]{Example}
\theoremstyle{remark}
\newtheorem{rem}[thm]{Remark}

\newtheorem{conj}[thm]{Conjecture}

\newcommand{\supp}{\text{supp}}

\newcommand{\R}{\mathbb{R}}

\newcommand{\N}{\mathbb{N}}
\newcommand{\Z}{\mathbb{Z}}

\newcommand{\bm}{\mathbf}

\newcommand{\nonemptyaty}{\textnormal{int}(\Delta^y(E))\neq \emptyset}

\newcommand{\Hdim}{\textnormal{dim}_{\mathcal{H}}}

\def\XXint#1#2#3{{\setbox0=\hbox{$#1{#2#3}{\int}$ }
\vcenter{\hbox{$#2#3$ }}\kern-.6\wd0}}

\title{Nonempty interior of pinned distance and tree sets}
\author{Tainara Borges, Benjamin Foster, Yumeng Ou, Eyvindur Palsson}

\address[T. Borges]{Department of Mathematics, Brown University, Providence, RI 02912}
\address[B. Foster]{Department of Mathematics, Stanford University, Stanford, CA 94305}
\address[Y. Ou]{Department of Mathematics, University of Pennsylvania, Philadelphia, PA 19104}
\address[E. Palsson]{Department of Mathematics, Virginia Tech, Virginia, VA 24061}

\begin{document}
\begin{abstract} 
For a compact set $E\subset\mathbb{R}^d$, $d\geq 2$, consider the pinned distance set $\Delta^{y}(E)=\lbrace |x-y| : x\in E\rbrace$. Peres and Schlag showed that if the Hausdorff dimension of $E$ is bigger than $\frac{d+2}{2}$ with $d\geq 3$, then there exists a point $y\in E$ such that $\Delta^{y}(E)$ has nonempty interior. In this paper we obtain the first non-trivial threshold for this problem in the plane, improving on the Peres--Schlag threshold when $d=3$, and we extend the results to trees using a novel induction argument.

\end{abstract}

\maketitle

\section{Introduction and Main results}

The Falconer distance problem asks how large the Hausdorff dimension of a compact set $E\subset\mathbb{R}^d$, $d\geq 2$, needs to be to guarantee that its \emph{distance set}
$$\Delta(E)=\lbrace |x-y|: x,y\in E \rbrace$$
has positive Lebesgue measure. Falconer \cite{Falconer85} showed through a construction that the dimensional threshold $\frac{d}{2}$ is the best one can hope for and proved that $\Hdim(E)>\frac{d}{2}+\frac{1}{2}$ is sufficient. Bridging this gap, namely solving Falconer's conjecture that the threshold $\frac{d}{2}$ is sufficient, is one of the major open problems at the interface of harmonic analysis and geometric measure theory. Despite much progress in recent years, this conjecture still remains open.

The pinned Falconer distance problem similarly asks how large the Hausdorff dimension of a compact set $E\subset\mathbb{R}^d$, $d\geq 2$, needs to be to guarantee that there exists a point $y\in E$ such that the \emph{pinned distance set} at $y$, 
$$\Delta^{y}(E)=\lbrace |x-y|: x\in E \rbrace,$$
has positive Lebesgue measure. Any dimensional threshold for the pinned Falconer distance problem immediately implies that the same dimensional threshold works for the unpinned original problem. This pinned variant was first studied by Peres and Schlag \cite{PS00} who obtained a sufficient threshold of $\Hdim(E)>\frac{d}{2}+\frac{1}{2}$. The original unpinned Falconer distance problem saw improvements by Bourgain \cite{Bourgain94}, Wolff \cite{Wolff99} and Erdo\u{g}an \cite{Erdogan05} that left the best threshold at $\Hdim(E)>\frac{d}{2}+\frac{1}{3}$ for a while. Liu \cite{Liu19} connected the main proof techniques for the pinned and unpinned variants and matched the threshold $\Hdim(E)>\frac{d}{2}+\frac{1}{3}$ for the pinned problem. Since Liu's work, all the best thresholds for the pinned and unpinned variants have been the same. With multiple improvements recently \cite{DZ19,GIOW20,DGOWWZ21,DIOWZ21,DORZ23} the best thresholds now stand at $\frac{5}{4}$ in the plane and at $\frac{d}{2}+\frac{1}{4}-\frac{1}{8d+4}$ when $d\geq 3$.

One of the motivations behind the Falconer distance problem is a classic result of Steinhaus \cite{Steinhaus20}, which states that if a set $E\subset\mathbb{R}^d$, $d\geq 1$, has positive Lebesgue measure, then the set $E-E = \lbrace x-y : x,y\in E\rbrace$ contains a neighborhood of the origin, which immediately implies that the distance set $\Delta(E)$ has positive Lebesgue measure. One can view the Falconer distance problem as reducing the assumptions needed on the set $E$ to obtain the conclusion on the distance set. In fact, from the Steinhaus result one not only concludes that the distance set has positive Lebesgue measure but also has nonempty interior. In that spirit, Mattila and Sj\"{o}lin \cite{MattilaSjolin} showed that a compact set $E\subset\mathbb{R}^d$, $d\geq 2$, satisfying $\Hdim(E)>\frac{d}{2}+\frac{1}{2}$ has the property that its distance set $\Delta(E)$ has nonempty interior. In the pinned setting Peres and Schlag \cite{PS00} showed that if $\Hdim(E)>\frac{d}{2}+1$ then there exists a point $y\in E$ such that the pinned distance set $\Delta^{y}(E)$ has nonempty interior. Despite all the recent progress on the Falconer distance problem and its pinned variant, neither of the results on nonempty interior has been improved since the original papers. Furthermore, observe that unlike when only asking for a positive Lebesgue measure, there is a gap here between the pinned and unpinned results. In the plane when $d=2$ one has $\frac{d}{2}+1=2$ so the result on nonempty interior for the pinned distance set in \cite{PS00} does not provide any information in that case. This leads us to the first main result of this paper, which establishes the first non-trivial threshold for nonempty interior of the pinned distance set in the plane.

\begin{thm}\label{Thm:pinnednonemptyin2dbasic} Let $E\subset \R^2$ be a compact set with $\Hdim(E)>\frac{7}{4}$. Then there exists $y\in E$ such that $\Delta^y(E)$ contains an interval.
\end{thm}

Many further variants of the Falconer distance problem exist. There are dimensional variants that obtain a lower bound on $\Hdim(\Delta(E))$ given $\Hdim(E)$, especially in the critical case of $\Hdim(E)=\frac{d}{2}$. See for example the recent works in \cite{DORZ23} and the breakthrough result of Shmerkin and Wang \cite{SW25} as well as the references therein. One can look at other notions of distance aside from the classical Euclidean one. Iosevich, Mourgoglou and Taylor \cite{IMT12} showed that the dimensional threshold $\Hdim(E)>\frac{d}{2}+\frac{1}{2}$ for nonempty interior obtained by Mattila and Sj\"{o}lin extends to a wide range of other metrics. Particularly important to this paper is a result by Iosevich and Liu \cite{IL19} in the pinned setting, which recovers the results of Peres and Schlag as well as improving on some dimensional variants and extending to a variety of metrics. Their Fourier analytic techniques rely on local smoothing inequalities and are very different than the projection type methods of Peres and Schlag. The framework of Iosevich and Liu relies on $L^2$ methods and thus only takes in local smoothing estimates that fit that framework. In this paper we upgrade their framework to handle more general $L^p\rightarrow L^p$ local smoothing estimates, which requires delicate arguments replacing simple $L^2$ orthogonality. Our next main theorem shows how the threshold for nonempty interior of the pinned distance set changes in $\mathbb{R}^d$ given different $L^p\rightarrow L^p$ local smoothing estimates. The hypothesis in the theorem that we can control a norm of a $\partial_t$ derivative will be used in conjuction with one-dimensional Sobolev embedding (in the $t$ variable) to prove a continuity result that implies the conclusion on nonempty interior of the pinned distance set.

\begin{thm}\label{Thm:pinnednonemptyassumingLS} 
Let $d\geq 2$ and $2<p<\infty$ fixed.
Assume that there exists $\eta>0$ such that the following local smoothing estimate holds
    \begin{equation}\label{partiallocalsmoothing}
\left(\int_{1}^{2}\int_{\R^d}|\partial^{\frac{1}{p}+\varepsilon}_tA_t(f)(x)|^pdx dt\right)^{1/p}\lesssim_{\varepsilon} \|f\|_{L^{p}_{-\eta+2\varepsilon}(\R^d)}
    \end{equation}
for every $\varepsilon>0$ sufficiently small.

Then for every compact set $E\subset \R^d$ with $\Hdim(E)>d-\eta$, there exists $y\in E$ such that $\Delta^y(E)$ contains an interval. More generally, given two compact sets $E,F$ in $\R^d$ satisfying 
$$\frac{1}{p'}\Hdim(E)+\frac{1}{p}\Hdim(F)>d-\eta,$$
there exists a probability measure $\mu_F$ in $F$ such that $\Delta^y(E)$ contains an interval for $\mu_F$-almost every $y\in F$. 
\end{thm}

If the local smoothing conjecture (see Conjecture \ref{conj: localsmoothingwaveeq}) were true in all dimensions, Theorem \ref{Thm:pinnednonemptyassumingLS} would yield the dimensional threshold $\frac{d}{2} + 1 - \frac{1}{2d}$, which would improve on the threshold $\frac{d}{2}+1$ obtained by Peres and Schlag in all dimensions. Theorem \ref{Thm:pinnednonemptyin2dbasic} realizes this optimal threshold when $d=2$. We however note that it is widely believed that the true threshold should be $\frac{d}{2}$, as in the original Falconer distance problem. To further show how robust our result is, we obtain an improvement on the Peres-Schlag dimensional threshold when $d=3$.

\begin{thm}\label{Thm:pinnednonemptyin3dbasic} Let $E\subset \R^3$ be a compact set with $\Hdim(E)>\frac{12}{5}$. Then, there exists $y\in E$ such that $\Delta^y(E)$ contains an interval.
\end{thm}

Although we don't recover the threshold $\frac{7}{3}$, that a solution of the local smoothing conjecture in $\mathbb{R}^{3+1}$ would yield, we improve on the Peres-Schlag threshold of $\frac{5}{2}$ and make it more than half the way towards the best threshold that our techniques could possibly yield.

Theorem \ref{Thm:pinnednonemptyassumingLS} is conditional on given local smoothing estimates. Using the latest results on local smoothing estimates we get the following unconditional corollary to Theorem \ref{Thm:pinnednonemptyassumingLS}. In the case $d=2$ we use the very recent sharp local smoothing inequalities in $\R^{2+1}$ proved in \cite{GWZ20} for the solutions of the wave equation in $\R^{2+1}$ or alternatively their extensions to Fourier integral operators in $\R^{2+1}$ obtained in \cite{GLMX23}. In the case $d \geq 3$, we use the even more recent results of \cite{GanWu}.

\begin{cor}\label{Cor:pinnednonemptyinhigherd}
Let $E,F\subset\mathbb{R}^d$ be compact sets satisfying 
\[\begin{cases}
    \frac{3}{4}\Hdim(E)+\frac{1}{4}\Hdim(F)>\frac{7}{4}, \text{ if }d=2,\\
    \frac{7}{10}\Hdim(E)+\frac{3}{10}\Hdim(F)>3-\frac{3}{5}=\frac{12}{5}, \text{ if }d=3, \text{ or }\\
    \frac{1}{p_d'}\Hdim(E)+\frac{1}{p_d}\Hdim(F)>d-\eta(p_d), \text{ for }p_d=2+\frac{8}{3d-4} \text{ and } \eta(p_d)=\frac{d-1}{p_d}\text{if }d\geq 4.
\end{cases}\]
Then there exists a probability measure $\mu_F$ in $F$ such that $\nonemptyaty$ for $\mu_F$ a.e $y\in E$. 
\end{cor}

Theorems \ref{Thm:pinnednonemptyin2dbasic} and \ref{Thm:pinnednonemptyin3dbasic} now follow from Corollary \ref{Cor:pinnednonemptyinhigherd} by simply taking $E=F$. To show that Corollary \ref{Cor:pinnednonemptyinhigherd} follows from Theorem \ref{Thm:pinnednonemptyassumingLS} we show that the operators that arise in our setting fit the setup from \cite{GWZ20} and \cite{GanWu} in Section \ref{sec: tools}. More precisely, Corollary \ref{Cor:pinnednonemptyinhigherd} follows from Theorem \ref{Thm:pinnednonemptyassumingLS} and Corollaries \ref{cor: sphereicallocalsmoothingd2} and \ref{cor: sphereicallocalsmoothinghigherd} in Section \ref{sec: tools}.

When $d\geq 4$ and if $E=F$ Corollary \ref{Cor:pinnednonemptyinhigherd} yields a threshold $\frac{d}{2}+\frac{7d-4}{6d}$ that matches the Peres-Schlag dimensional threshold when $d=4$ and is worse when $d\geq 5$. Our result is more flexible with two different sets, $E$ and $F$, and immediately implies the following exceptional set estimates for the bad pins of $E$.

\begin{cor}
    Let $E\subset \R^d$ be a compact set. Then 
     \[\Hdim(\{x\in \R^d\colon \textnormal{int}(\Delta^y(E))=\emptyset\})\leq\begin{cases} 7-3\Hdim(E), \text{ if }d=2,\\
       8-\frac{7}{3}\Hdim(E), \text{ if }d=3,\\
    \frac{3d^2+7d-4}{3d-4}-\frac{3d+4}{3d-4}\Hdim(E)\text{ if }d\geq 4.
    \end{cases}\]
\end{cor}

It is worth comparing the corollary above with the corresponding corollary in Iosevich and Liu paper \cite[Corollary 1.4]{IL19},
which gives
\[\Hdim(\{x\in \R^d\colon \text{int}(\Delta^y(E))=\emptyset\})\leq
    \begin{cases}
      8-2\Hdim(E)\text{ if }d=3,\\
      \frac{(d+1)^2}{d-1}-\frac{d+1}{d-1}\Hdim(E), \text{ if }d\geq 4, 
    \end{cases} 
    \]
and does not provide information when $d=2$. Our exceptional set estimates improve upon those of Iosevich--Liu in dimension $d=2$ (where they obtain no result) and in dimension $d=3$. For $d \ge 4$, our bound is sharper precisely in the regime $\Hdim(E) > d - 3 + \frac{4}{d}$.

For a compact set $E\subset\mathbb{R}^d$ one can consider a generalization of the distance set to a tree $\mathcal{T}$, denoted by $\Delta_{\mathcal{T}}(E)$, see precise definition in Section \ref{sec trees}. Bennett, Iosevich and Taylor showed in \cite{BIT16} for particular trees, called chains, that if $\Hdim(E)>\frac{d+1}{2}$ then $\Delta_{\mathcal{T}}(E)$ has nonempty interior. This result was extended to all finite trees by Iosevich and Taylor \cite{IT19}. In a similar spirit, one can also consider a pinned variant $\Delta_{\mathcal{T},v_0}^{x_0}(E)$ where the tree has the distinguished vertex $v_0$ pinned at $x_0\in E$. The third listed author and Taylor \cite{OuTaylor} showed that when $d=2$ if $\Hdim(E)>\frac{5}{4}$ the Lebesgue measure of $\Delta_{\mathcal{T},v_0}^{x_0}(E)$ is positive, for some $x_0\in E$. No results exist in the literature on nonempty interior for pinned trees in any dimension. Our second main result provides the first result in the plane in this direction.

\begin{thm}\label{Thm:pinnednonemptytreein2d}
    Let $E\subset \R^2$ be a compact set with $\Hdim(E)>7/4$. Let $(\mathcal{T},v_0)$ be any finite tree with $n+1$ vertices, including a distinguished pinned vertex $v_0$. Then there exists a point $x_0\in E$ such that $\Delta_{\mathcal{T},v_0}^{x_0}(E)$ has nonempty interior.
\end{thm}

The previous results on trees \cite{BIT16,IT19,OuTaylor} all used inductive strategies, but in the case of nonempty interior showed boundedness of the density as well as lower bounds and then concluded from direct approximation arguments that the density had to be continuous. Even in the most recent paper of Greenleaf, Iosevich and Taylor \cite{GIT24preprint}, which utilizes the strong microlocal analysis framework developed by the authors in previous papers \cite{GIT21,GIT22,GIT24}, they still follow a similar scheme and obtain the same dimensional threshold $\frac{d+1}{2}$. Our approach is built around showing boundedness for derivatives of the densities and utilizing local smoothing estimates. In order to make this work, we need to develop a novel induction strategy that handles all the derivatives appropriately. If implemented naively, taking all the required derivatives would lead to a much worse threshold that in most cases would be trivial. Our methods are robust, as in the case for distances, and extend to higher dimensions and can showcase the reliance on the $L^p\rightarrow L^p$ local smoothing estimates. This is the content of our next main theorem.

\begin{thm}\label{Thm:pinnednonemptytreeinanydimensionwithlocalsmoothing}
Let $d\geq 2$ and $2<p<\infty$ fixed.
Assume that there exists $\eta>0$ such that a local smoothing estimate (\ref{partiallocalsmoothing}) holds for every $\varepsilon>0$ sufficiently small.
Let $E\subset \R^d$ be a compact set with $\Hdim(E)>d-\eta$. Let $(\mathcal{T},v_0)$ be any finite tree with $n+1$ vertices, including a distinguished pinned vertex $v_0$. Then there exists a point $x_0\in E$ such that $\Delta_{\mathcal{T},v_0}^{x_0}(E)$ has nonempty interior.
\end{thm}

Using the best-known $L^p\rightarrow L^p$ local smoothing estimates as stated in Corollaries \ref{cor: sphereicallocalsmoothingd2} and \ref{cor: sphereicallocalsmoothinghigherd} in Section \ref{sec: tools}, coming from \cite{GWZ20} and \cite{GanWu}, this yields Theorem \ref{Thm:pinnednonemptytreein2d} in the plane and the following estimates in higher dimensions, which are all new in the literature.

\begin{thm}\label{Thm:pinnednonemptytreeinhigherdim}
    Let $E\subset \R^d$, $d\geq 3$ be a compact set with $\Hdim(E)>\frac{12}{5}$ when $d=3$ or $\Hdim(E)>\frac{d}{2}+\frac{7d-4}{6d}$ when $d\geq 4$. Let $(\mathcal{T},v_0)$ be any finite tree with $n+1$ vertices, including a distinguished pinned vertex $v_0$. Then there exists a point $x_0\in E$ such that $\Delta_{\mathcal{T},v_0}^{x_0}(E)$ has nonempty interior.
\end{thm}

If the local smoothing conjecture (see Conjecture \ref{conj: localsmoothingwaveeq}) were to be true when $d\geq 3$ at the critical exponent $p=\frac{2d}{d-1}$, the threshold would improve to $\frac{d+2}{2}-\frac{1}{2d}$. Note that all our results for trees hold with the same dimensional threshold for any finite tree. We finally remark that our results on trees could be stated even more generally, drawing points from different sets in the spirit of Corollary \ref{Cor:pinnednonemptyinhigherd}. For clarity of exposition, we leave the results as they are and leave the details to the interested reader.

Beyond distances and trees, triangles and simplices are perhaps the point configurations that have received the most attention. In this setting there are a number of results on obtaining positive Lebesgue measure of the configuration sets, such as \cite{EHI13,GI12,GILP15}, one that obtains positive Lebesgue measure for pinned configurations \cite{IPPS22preprint}, and a few that obtain nonempty interior of the configuration set \cite{PR23,PR25,GIT22,GIT24,GIT24preprint}. The only results we are aware of that obtain pinned nonempty interior for point configurations with $3$ or more points come from the sequence of papers by Greenleaf, Iosevich and Taylor \cite{GIT22,GIT24} where some particular configurations are handled, such as volumes.

Here is an overview of the organization of the paper. In Section \ref{sec: tools} we present tools used throughout the paper and give preliminaries on local smoothing, including Corollaries \ref{cor: sphereicallocalsmoothingd2} and \ref{cor: sphereicallocalsmoothinghigherd} that translate the current best local smoothing results from \cite{GWZ20} and \cite{GanWu} to our setup. In Section \ref{sec:maindisttheorem} we prove Theorem \ref{Thm:pinnednonemptyassumingLS}, our main result on distances. In Section \ref{sec:toycases} we show a couple of illustrative toy cases for how our induction argument works for trees. In Section \ref{sec trees} we present the proof of Theorem \ref{Thm:pinnednonemptytreeinanydimensionwithlocalsmoothing}, our main result on trees.

\subsection*{Acknowledgments}
The authors are grateful to the anonymous referee for their helpful comments. Y.O. is supported in part by NSF DMS-2142221 and NSF DMS-2055008.

\section{Tools and Notation}\label{sec: tools}

\subsection{Littlewood-Paley pieces and Sobolev spaces}
By $P_jf$  we denote the $j$th Littlewood-Paley projection of $f$, which we sometimes denote as $f^j$ for short. More explicitly, take $\beta_0\in C^{\infty}_c(\R)$ to be a nonnegative function such that $\beta_0(s)\equiv1$ if $|s|\leq 1$ and
$\beta_0=0$ if $|s|\geq 2$. Let $\beta(s):=\beta_0(s)-\beta_0(2s)$ so that $\beta_j(s):=\beta(2^{-j}s)$ is supported in $\{s\colon 2^{j-1}<|s|<2^{j+1}\}$ and 
$\beta_0(s)+\sum_{j=1}^{\infty} \beta(2^{-j}s)\equiv 1.$ Then
\begin{equation}\label{Littlepaley}
    \mathcal{F}(P_0f)(\xi)=\beta_0(|\xi|)\hat{f}(\xi) \text{ and } \mathcal{F}(P_jf )(\xi)=\beta(2^{-j}|\xi|)\hat{f}(\xi) \text{ for all } j\geq 1.
\end{equation}

\noindent For us, 
$\mathcal{F}(f)(\xi):=\hat{f}(\xi)=\int_{\R^d}  f(x)e^{-2\pi ix\cdot \xi}dx$.

We use the following notation for Sobolev spaces. Given  $\gamma\in \R$, define 
$$\|f\|_{L^{p}_{\gamma}(\R^d)}:=\|(\langle \cdot \rangle^{\gamma}\hat{f})^{\vee}\|_{L^{p}(\R^d)}\text{ where }\langle \xi\rangle:=(1+|\xi|^2)^{1/2},$$
and $L_{\gamma}^p(\R^d)=\{f\in \mathcal{S}'(\R^d)\colon (\langle \cdot\rangle^{\gamma} \hat{f})^{\vee}\in L^p\}$.

\begin{rem} \label{rem drop smooth factor}
    Frequently in the sequel, we will encounter the scenario where a function $\varphi$ in a Sobolev space is multiplied by a smooth function that is comparable to $1$ on the support of $\varphi$ (most often a polynomial factor where $\varphi$ is supported on a compact interval away from $0$). We will generally drop such factors without further comment as the Sobolev norm of the product is at most a bounded factor larger than the Sobolev norm of $\varphi$.
\end{rem}

When $\gamma>d/p$ then every function in $L^{p}_{\gamma}(\R^d)$ can be modified in a set of measure zero so that the resulting function is H\"older continuous (see \cite[Section 2.8.1]{Triebel95} for example, where one just needs to recall that $L^{p}_{\gamma}(\R^d)=F_{p,2}^{\gamma}(\R^d)$). We will only need this fact in $d=1$, so we state it below for future reference.

\begin{thm}\label{tool: embeddingthm}
Let $1<p<\infty$, and $0<\alpha:=\gamma-1/p<1$. Then every function in $L^{p}_{\gamma}(\R)$ is $\alpha$-H\"older continuous with
\begin{equation}\label{ineq: embedding}
    \|f\|_{C^{\alpha}(\R)}=\sup_{t\in \R}|f(t)|+\sup_{t\neq s}\frac{|f(t)-f(s)|}{|t-s|^{\alpha}}\leq C \|f\|_{L^{p}_{\gamma}(\R)}.
\end{equation}

\end{thm}

If $\varphi$ is a function of $t\in \R$ we will denote $\partial_t^{\gamma}\varphi(t):=(\langle \tau\rangle^{\gamma} \hat{\varphi}(\tau))^{\vee}(t)$ so that $\|\varphi\|_{L^{p}_{\gamma}(\R)}=\|\partial_t^{\gamma}\varphi\|_{L^p(\R)}$. Although we use the partial derivative notation to denote this fractional derivative, we will make use of it in the sequel as a way to access \textbf{one-dimensional} Sobolev embedding, allowing us to show that certain quantities are continuous.

\subsection{Restricted measures}\label{subsection: restrictmeasures}
Let $\mu$ be a Borel probability measure supported in a compact set $E\subset \R^d$. Given $A\subset E$, with $\mu(A)>0$ we will refer to the  \textbf{restriction} of $\mu$ to $A$ as the probability measure supported in $A$ given by 
\begin{equation}\label{def: restrictionmeasure}
\mu_A(S)=\frac{\mu(S\cap A)}{\mu(A)},\, \text{ for any Borel set }S \text{ in } \R^d.
\end{equation}

If $\mu(A)=1-\epsilon$, for some $\epsilon\in (0,1)$ we will say that $\mu_A$ is an $(1-\epsilon)$-restriction of $\mu$.

\subsection{Preliminaries in local smoothing}

Let $u$ be a solution of the Cauchy problem 
\begin{equation}\label{cauchywaveeq}
    \begin{split}
        \begin{cases}
            (\partial_t^2-\Delta)u(x,t)=0,\, (x,t)\in \R^{d}\times \R \\
            u(x,0)=f(x)\\
            \partial_t u(x,0)=0.
        \end{cases} 
    \end{split}
\end{equation}

For each fixed time $t>0$, by a classical result of Peral \cite{Peral80} and Miyachi \cite{Miyachi80}, in the range $2<p<\infty$ we have that 
\begin{equation}\label{localsmoothingforfixedt}
 \|u(\cdot, t)\|_{L^p(\R^d)}\lesssim_{t,p} \|f\|_{L_{s_p}^{p}(\R^d)}   
\end{equation}
where $s_p=\frac{d-1}{2}-\frac{d-1}{p}$.

Sogge conjectured \cite{Sogge91} that by taking $L^p$ norms in both space and local time, one has the following local smoothing estimate for the wave equation.

\begin{conj}[Local smoothing conjecture for the wave equation in $\R^{d+1}$\,\cite{Sogge91}]\label{conj: localsmoothingwaveeq}
    Let $u$ be a solution of the Cauchy problem in (\ref{cauchywaveeq}). Then for all $p\geq \frac{2d}{d-1}$ and $s_p=\frac{d-1}{2}-\frac{d-1}{p}$ we have 
    \begin{equation}\label{conjlocalsmoothingwave}
\|u\|_{L^p(\R^d\times[1,2])}\lesssim_{\varepsilon} \|f\|_{L^p_{s_p-1/p+\varepsilon}(\R^d)}.
    \end{equation}
    
\end{conj}

In other words, by also integrating locally in time, there is almost an extra $1/p$ smoothing in comparison to the fixed time estimate in (\ref{localsmoothingforfixedt}). In a groundbreaking result by Guth, Wang and Zhang \cite{GWZ20}, the local smoothing conjectured for the wave equation was confirmed for $d=2$. We state it below.

\begin{thm}\cite{GWZ20}\label{Thm: sharplocalsmoothingind2}
    Suppose that $u(x,t)$ is a solution of the wave equation in $\R^{2+1}$ with initial data $u(x,0)=f$ and $\partial_tu(x,0)=0$. Then for all $p\geq 4$ and $\varepsilon>0$,
    \begin{equation}
        \|u\|_{L^{p}(\R^2\times [1,2])}\lesssim_{\varepsilon} \|f\|_{L^p_{1/2-2/p+\varepsilon}(\R^2)}.
    \end{equation}
\end{thm}

Recall that a solution of the Cauchy problem in (\ref{cauchywaveeq}) can be written in terms of the half-wave propagator as 
$$u(x,t)=e^{it\sqrt{-\Delta}}f(x):=\int_{\R^d} e^{2\pi i(x\cdot\xi+t|\xi|)}\hat{f}(\xi)\,d\xi. $$

For every $t>0$ define the spherical average of $f$ at $x$ of radius $t>0$ as 
\begin{equation}\label{def: averageAt}
    A_t(f)(x)=\int_{S^{d-1}} f(x-ty)\,d\sigma(y)=f*\sigma_t(x),
\end{equation}
where $\sigma$ is the normalized canonical measure in the unit sphere $S^{d-1}=\{x\in \R^d\colon |x|=1\}$. When written as a Fourier multiplier we have 
$$A_tf(x)=\int_{\R^d} \hat{f}(\xi) \hat{\sigma}(t\xi)e^{2\pi i x\cdot \xi}d\xi.$$

It is possible to transfer local smoothing estimates for the wave equation to local smoothing estimates for spherical averages. For example, that is shown in \cite{BORSS22} (See Section 2.3 and Proposition 4.4). In what follows, for the reader's convenience, we'll explain in more detail how that can be done in our context. 

One can write
$$\hat{\sigma}(\xi)=b_0(|\xi|)+\sum_{\pm} b_{\pm}(|\xi|)e^{\pm 2\pi i |\xi|}$$
where $b_0\in C^{\infty}_{c}(\R)$ is supported in $[-1,1]$ and each $b_{\pm}$ is supported in $|\xi|\geq 1/2$ with $b_{\pm}\in S^{-(d-1)/2}(\R)$, i.e., $|\frac{d^k (b_{\pm})}{ds^k}(s)|\lesssim_k (1+|s|)^{-\frac{d-1}{2}-k}$. Therefore, for $j\geq 1$, and $t\in [1,2]$

\begin{equation}\label{representationofAinwavepieces}
    A_t(f^j)(x)=\sum_{\pm} \int \beta_j(|\xi|)b_{\pm}(t|\xi|)\hat{f}(\xi)e^{2\pi i(x\cdot \xi\pm it|\xi|)}d\xi.
\end{equation}

Define for $j\geq 0$, $$u^j_{\pm}(x,t):=\int \beta_j(|\xi|)\hat{f}(\xi)e^{2\pi i (x\cdot \xi\pm it |\xi|)}d\xi.$$
Assume the local smoothing conjecture is true for the wave equation in $\R^{d+1}$ as in (\ref{conjlocalsmoothingwave}). Then for $p\geq \frac{2d}{d-1}$,
\begin{equation}\label{ineq: waveLSdecay}
    \|u^{j}_{\pm}\|_{L^{p}(\R^d\times [1,2])}\lesssim_{\varepsilon} 2^{j(\frac{d-1}{2}-\frac{d}{p}+\varepsilon)}\|f^j\|_{L ^p(\R^d)}.
\end{equation}
We claim that the inequality above implies
\begin{equation}\label{ineq: sphericalaveragedecay}
    \|A_t(f^j)(x)\|_{L^p(\R^d\times [1,2])}\lesssim_{\varepsilon} 2^{j(-d/p+\varepsilon)}\|f^j\|_{L^p(\R^d)}.
\end{equation}

Indeed, note that 
\begin{equation}
\begin{split}
    T_{b_{\pm}}(f^j)(x,t):=&\int \beta_j(|\xi|)b_{\pm}(t|\xi|)\hat{f}(\xi)e^{2\pi i(x\cdot \xi\pm it|\xi|)}d\xi\\
    =&\int \left(\tilde{\beta}_j(|\xi|)b_{\pm}(t|\xi|)\right)\cdot \left(\beta_j(|\xi|)\hat{f}(\xi)e^{\pm 2\pi it |\xi|}\right)e^{2\pi ix\cdot \xi}d\xi\\
    =& \mathcal{F}^{-1}[\tilde{\beta}_j(|\cdot|)\ b_{\pm}(t|\cdot|)]*u^{j}_{\pm}(\cdot, t)(x)
\end{split}
\end{equation}
where $\tilde{\beta_j}:=\beta_{j-1}+\beta_j+\beta_{j+1}\equiv 1$ in the support of $\beta_j$.
From Young's inequality 
$$\|T_{b_{\pm}}f^j(x,t)\|_{L^p_x}\lesssim \|u^j_{\pm}(x,t)\|_{L^{p}_x}\|\mathcal{F}^{-1}[\tilde{\beta}_j(|\xi|)\ b_{\pm}(t|\xi|)](x)\|_{L^{1}_x}.$$
By a standard integration by parts argument and since $b_{\pm}\in S^{-\frac{d-1}{2}}(\R)$ we have $$\|\mathcal{F}^{-1}[\tilde{\beta}_j(|\xi|)\ b_{\pm}(t|\xi|)](x)\|_{L^{1}_x}\leq C2^{-j\frac{d-1}{2}},$$ uniformly for $t\in [1,2]$. By further integrating the inequality above in $t\in [1,2]$ and using the local smoothing in (\ref{ineq: waveLSdecay}) we get
$$\sum_{\pm}\|T_{b_{\pm}}f^j(x,t)\|_{L^p(\R^d\times [1,2])}\lesssim 2^{-j\frac{d-1}{2}}\|f^j\|_{L^p_{s_p-1/p+\varepsilon}(\R^d)}\lesssim 2^{j(-d/p+\varepsilon)}\|f^j\|_p,$$
which leads to the desired inequality (\ref{ineq: sphericalaveragedecay}). More generally, for all $k\in \Z_{+}$, one would have,
$$\left\|\frac{d^k}{dt^k}A_t(f^j)(x)\right\|_{L^p(\R^d\times [1,2])}\lesssim_{\varepsilon,k} 2^{j(k-d/p+\varepsilon)}\|f^j\|_{L^p}$$ 
which, combined with interpolation, would lead to the estimates in the conjecture below.

\begin{conj}[$L^p\rightarrow L^p$ local smoothing conjecture for spherical averages]\label{conj: sphericalcasedimd}
    Let $d\geq 2$ and $\gamma\geq 0$. For all $p\geq \frac{2d}{d-1}$ one has 
    \begin{equation}
\left(\int_{1}^{2}\int_{\R^d}|\partial^{\gamma}_tA_t(f)(x)|^pdxdt\right)^{1/p}\lesssim_{\varepsilon} \|f\|_{L^{p}_{\gamma-\frac{d}{p}+\varepsilon}(\R^d)}.
    \end{equation}
\end{conj}

From Theorem \ref{Thm: sharplocalsmoothingind2} one gets the conjectured local smoothing for spherical averages for $d=2$.

\begin{cor}[$L^{p}\rightarrow L^{p}$ local smoothing for spherical averages in $\R^{2+1}$]\label{cor: sphereicallocalsmoothingd2}
    Let $p\geq 4$ and $\gamma\geq 0$. Then 
    \begin{equation}
\left(\int_{1}^{2}\int_{\R^2}|\partial^{\gamma}_tA_t(f)(x)|^pdxdt\right)^{1/p}\lesssim_{\varepsilon} \|f\|_{L^{p}_{\gamma-\frac{2}{p}+\varepsilon}(\R^2)}.
    \end{equation}
    In particular, for all $\varepsilon>0$ one has 
     \begin{equation}
\left(\int_{1}^{2}\int_{\R^2}|\partial^{1/4+\varepsilon}_tA_t(f)(x)|^4dxdt\right)^{1/4}\lesssim_{\varepsilon} \|f\|_{L^{p}_{-\frac{1}{4}+2\varepsilon}(\R^2)}.
    \end{equation}
\end{cor}

The local smoothing result above implies that when $d=2$, the hypotheses of Theorem \ref{Thm:pinnednonemptyassumingLS} are satisfied with $p=4$ and $\eta=1/4$ leading to the unconditional Corollary \ref{Cor:pinnednonemptyinhigherd}.

For $d\geq 3$, we make use of a recent improvement on local smoothing estimates for the wave equation in \cite[Theorem 1.3]{GanWu} to get the following corollary.

\begin{cor}[$L^{p}\rightarrow L^{p}$ local smoothing for spherical averages in $\R^{d+1}$]\label{cor: sphereicallocalsmoothinghigherd}
    Let $d\geq 3$ and $p\geq p_d$ where $p_d=\frac{10}{3}$ in $d=3$ and $p_d=2+\frac{8}{3d-4}$ for $d\geq 4$.  
    Then for any $\gamma\geq 0$, 
    \begin{equation}
\left(\int_{1}^{2}\int_{\R^d}|\partial^{\gamma}_tA_t(f)(x)|^pdxdt\right)^{1/p}\lesssim_{\varepsilon} \|f\|_{L^{p}_{\gamma-\frac{d}{p}+\varepsilon}(\R^d}),
    \end{equation}
In particular, for $p=p_d$ and $\gamma=1/p_d+\varepsilon$,
\begin{equation}
\left(\int_{1}^{2}\int_{\R^d}|\partial^{\frac{1}{p_d}+\varepsilon}_tA_t(f)(x)|^{p_d}dxdt\right)^{1/p_d}\lesssim_{\varepsilon} \|f\|_{L^{p_d}_{-\frac{d-1}{p_d}+2\varepsilon}(\R^d)}.
    \end{equation}
   
\end{cor}

Observe that for $p_d$ as in the corollary above one has $d-\frac{d-1}{p_d}=\frac{12}{5}$ for $d=3$ and $d-\frac{d-1}{p_d}=\frac{d}{2}+\frac{7d-4}{6d}$ when $d\geq 4$, the thresholds showing up in Theorem \ref{Thm:pinnednonemptytreeinhigherdim}.

\section{Proof of Theorem \ref{Thm:pinnednonemptyassumingLS}}\label{sec:maindisttheorem}

Take $E,F$ to be compact sets in $\R^d$ satisfying 
$$\frac{1}{p'}\Hdim(E)+\frac{1}{p}\Hdim(F)>d-\eta.$$

Denoting $s_{E}=\Hdim(E)$, there exists $\mu_E$, a Frostman probability measure in $E$, such that 
\begin{equation}\label{Frostman}
    \mu_{E}(B(x,r))\lesssim_{\varepsilon} r^{s_E-\varepsilon}\text{ for all }r>0 \text{ and }x\in \R^d,
\end{equation}
where $\varepsilon>0$ will be chosen sufficiently small. Similarly, let $s_F=\dim_{\mathcal{H}}(F)$ and take $\mu_F$ to be a Frostman probability measure in $F$ satisfying (\ref{Frostman}) with $E$ replaced by $F$.

A standard argument along the lines of \cite[Lemma 3]{BIO23} allows us to assume that all distances from points in $E$ to points in $F$ live in some compact interval $[a,b]$ with $a>0$. Really, what we are doing is finding positive measure subsets of these sets that are separated and working on these subsets instead. The interval $[a,b]$ can be decomposed into finitely many intervals at different dyadic scales, and we can apply the argument in each interval separately. So we can assume without loss of generality that $[a,b]=[1,2]$ and $dt$ is Lebesgue measure in $[1,2]$.

 If $\mu_E$ is a Frostman measure in $E$, one can define a measure $\tilde{\nu}_{E}^y$ in $\R$ as
$$\int f(t)d\tilde{\nu}^y_E(t)=\int f(|x-y|)d\mu_E(x).$$

In other words, $\tilde{\nu}_{E}^{y}=(\Phi^y)_{*}(\mu_E)$, is the pushforward of the measure $\mu_E$ through the map $\Phi^{y}:E\rightarrow \R$, mapping $x\mapsto |x-y|$.
Note that $\tilde{\nu}^y_E$ is supported in $\Delta^y(E)$, which is contained in $[1,2]$ by assumption. We want to show that for $\mu_F$-almost every $y$, $\tilde{\nu}^y_{E}$ is absolutely continuous with respect to the Lebesgue measure in $\R$, and it admits a continuous density. In that case, since $\tilde{\nu}_{E}^y$ is a probability measure with support in $[1,2]$ and continuous density, it follows that its support contains an interval, and so does $\Delta^{y}(E)$.

 By the Sobolev embedding in $\R$ (Theorem \ref{tool: embeddingthm}), if $\gamma>1/p$ then every element of $L_{\gamma}^p(\R)$ can be modified in a set of measure zero so that the resulting function is bounded and uniformly continuous. Hence, in order to prove Theorem \ref{Thm:pinnednonemptyassumingLS} it is enough to show that if $E,\,F$ are compact sets in $\R^d$ such that $\frac{1}{p'}\Hdim(E)+\frac{1}{p}\Hdim(F)>d-\eta,$
where $\eta$ is the assumed local smoothing as in (\ref{partiallocalsmoothing}) then there exist Frostman probability measures $\mu_E,\,\mu_F$ in $E,\,F$ respectively and $\gamma>1/p$ such that
\begin{equation}\label{reduction1}
  \left(\int \|\tilde{\nu}^y_E(t)\|_{L^{p}_{\gamma}(dt)}^pd\mu_F(y)\right)^{1/p}<\infty.
\end{equation}
Such an estimate implies that for $\mu_F$ a.e. $y$ one has that $\|\tilde{\nu}_{E}^{y}\|_{L_{\gamma}^{p}([1,2])}<\infty$, and consequently $\Delta^y(E)$ has nonempty interior.

We can rewrite (\ref{reduction1}) as a mixed norm estimate of the form
\begin{equation}\label{reduction2}
\|\partial_{t}^{1/p+\varepsilon}\tilde{\nu}_{E}^y(t)\|_{L^{p}(d\mu_F(y) dt )}<\infty,
\end{equation}
for some $\varepsilon>0$ (i.e. by taking $\gamma=1/p+\varepsilon$).

Let $\mu_E^{\delta}:=\mu_E*\rho^{\delta}$ where $\{\rho_{\delta}\colon \delta>0\}$ is an approximate identity. Namely, choose a nonnegative $C^{\infty}$ function $\rho$ with $\text{supp}(\rho)\subset B(0,1),\,\int \rho(t)dt=1$ and let $\rho_{\delta}(x)=\delta^{-d}\rho(x/\delta)$. Note that for each $\delta>0$ one has $\mu_{E}^{\delta}\in C^{\infty}_c$ and $\|\mu_{E}^{\delta}\|_{L^1}\lesssim 1$.

Let $(\tilde{\nu}_{E}^{y})^{\delta}$ be the pushforward of $\mu_E^{\delta}$ through the map $\Phi^y$. One can check that for each $\delta>0$ $$(\tilde{\nu}_{E}^y)^{\delta}(t)=t^{d-1}A_t(\mu_E^{\delta})(y)$$
where $A_t$ denotes the spherical averaging operator in (\ref{def: averageAt}). Note that by Remark \ref{rem drop smooth factor}, we will be able to drop the $t^{d-1}$ factor in the sequel.

For each $y$, one has $(\tilde{\nu}_{E}^y)^{\delta}\rightharpoonup \tilde{\nu}_{E}^y$ in $\mathcal{S}'(\R_{+})$.
Moreover, if we recall that $\partial^{\gamma}_t\tilde{\nu}_{E}^y=(\langle\tau\rangle^{\gamma}(\tilde{\nu}_{E}^{y}\hat{)}(\tau))^{\vee} $ understood in a distributional sense, then for all $\varphi\in \mathcal{S}(\R^d\times \R_{+})$
\begin{equation}
    \begin{split}
    \lim_{\delta\rightarrow 0}\langle\partial^{\gamma}_t(\tilde{\nu}_{E}^y)^{\delta},\varphi\rangle=&\lim_{\delta\rightarrow 0}\int \langle\tau\rangle^{\gamma}\mathcal{F}(\tilde{\nu}_{E}^y)^{\delta}(\tau)\mathcal{F}^{-1}({\varphi(y,\cdot)})(\tau)d\tau d\mu_F(y)\\
    =&\lim_{\delta\rightarrow 0}\int \langle\tau\rangle^{\gamma}\mathcal{F}(\tilde{\nu}_{E}^y)(\tau)\hat{\rho}(\delta \tau)\mathcal{F}^{-1}(\varphi(y,\cdot))(\tau)d\tau d\mu_F(y)\\
    =&\int \langle\tau\rangle^{\gamma}\mathcal{F}(\tilde{\nu}_{E}^y)(\tau)\mathcal{F}^{-1}(\varphi(y,\cdot))(\tau)d\tau d\mu_F(y)=\langle \partial^{\gamma}_t(\tilde{\nu}_{E}^y),\varphi\rangle.
    \end{split}
\end{equation}

That is, it holds that
\begin{equation}\label{weakconvergence}
\partial^{\gamma}_t(\tilde{\nu}_{E}^{y})^{\delta}(t)\rightharpoonup \partial^{\gamma}_t\tilde{\nu}_E^y(t)\text{ in }\mathcal{S}'(\R^d\times \R_{+}). 
\end{equation}

The function $t\mapsto t^{d-1}$ is smooth and very well behaved away from the origin so by defining 
$(\nu_E^y)^{\delta}(t)=A_t(\mu_E^{\delta})(y)$ we still have weak convergence to a finite nonzero Borel measure $\nu_E^y$ supported in $\Delta^y(E)$ (and there is also weak convergence for the fractional $\gamma$ derivatives).

Suppose that we prove that 
\begin{equation}\label{thickened2}
\|\partial_{t}^{1/p+\varepsilon}(\nu_E^y)^{\delta_1}(t)\|_{L^{p}(\mu_F^{\delta_{2}} \times dt )}\lesssim 1,
\end{equation}
with implicit constant independent of $\delta_1,\delta_2>0$. Then, sending $\delta_2$ to $0$ we get 
\begin{equation}\label{thickened1}
\|\partial_{t}^{1/p+\varepsilon}(\nu_E^y)^{\delta_1}(t)\|_{L^{p}(d\mu_F \times dt )}\lesssim 1,
\end{equation}
uniformly in $\delta_1>0$. Then by using that $\mathcal{S}(\R^{n}\times \R)$ is dense in $L^{p}(d\mu_F dt)$ and the norm in $L^{p}(d\mu_F dt)$ is realized by dualization against Schwartz functions (see \cite[Page 270]{Mattilabook2015}) combined with the weak convergence in (\ref{weakconvergence}) we recover the desired estimate (\ref{reduction2}) where $\delta_1,\delta_2$ are dropped.

To keep the notation lighter, we will drop the mollifications in the arguments below. Still, the reader should keep in mind that we can assume that a measure $\mu$ was replaced with their smooth mollifications $\mu^{\delta}$, and if there is more than one compact set involved say $E_1,\,E_2,\,\dots,\, E_m$ one could replace $\mu_{E_i}$ with $\mu_{E_i}^{\delta_i}$. The bounds we prove will be independent on $\delta_i$s so limiting arguments similar to the one above apply. 

We now prove (\ref{reduction2}) with a duality strategy inspired by \cite[Theorem 4.1]{IKSTU}.

Let $P_j$ be the $j$th Littlewood-Paley projection piece of $f$, as in (\ref{Littlepaley}). By duality, it is enough to show that for any $g\in L^{p'}(d\mu_F dt)$ with  $\|g\|_{L^{p'}(d\mu_F\times dt)}=1$ one has that
\begin{equation}
  \left|  \int_{1}^{2}\int \partial_t^{1/p+\varepsilon}\nu_E^{y}(t) g(y,t)d\mu_F(y)dt\right|\lesssim 1,
\end{equation}
with implicit constant independent of $g$. Decomposing into Littlewood-Paley pieces, we have 

\begin{equation}
    \begin{split}
        \int_{1}^{2}\int &\partial_t^{1/p+\varepsilon}\nu_E^{y}(t) g(y,t)d\mu_F(y)dt\\
        = & \sum_{j=0}^{\infty} \int_{1}^{2}\int \partial_t^{1/p+\varepsilon}A_t(P_j\mu_{E})(y)g_t(y)d\mu_F(y)dt,
    \end{split}
\end{equation}
 where we used the notation $g_t(y)=g(y,t)$. By Plancherel, we can replace the measure function $g_t\mu_F$ with a Littlewood piece as well, say $\tilde{P_j}(g_t d\mu_F)$ for $\tilde{P}_j=\sum_{i=-2}^{2}P_{j+i}$. Next, we do
\begin{equation}\label{computationthm1}
    \begin{split}
       \sum_{j=0}^{\infty}& \int_{1}^{2}\int \partial_t^{1/p+\varepsilon}A_t(P_j\mu_{E})(y) \tilde{P}_j(g_td\mu_F)(y)dy dt\\
       \lesssim & \sum_{j\geq 0} \|\partial_t^{1/p+\varepsilon}A_{t}(P_j\mu_E)(y)\|_{L^{p}(dy dt)}\|\tilde{P}_j(g_t \mu_F)(y)\|_{L^{p'}(dy dt)}\\
       \lesssim & \sum_{j\geq 0} \| \mu_E^j\|_{L^{p}_{-\eta+2\varepsilon}}\|\tilde{P}_j(g_t \mu_F)(y)\|_{L^{p'}(dy dt)}
    \end{split}
\end{equation}
where we used the assumed local smoothing (\ref{partiallocalsmoothing}). 

Next, we will need some standard estimates. First, we recall the following lemma from  \cite[Lemma 2.4]{IL19} that will be used multiple times in this paper.

\begin{lem}[\cite{IL19}]\label{IL lemma}
Let $1\leq p\leq \infty$, $\gamma\in \R$ and $\mu $ Frostman measure in $\R^d$ satisfying $\mu(B(x,r))\leq C_{\varepsilon} r^{s-\varepsilon}$ for all $x\in \R^d$ and $r>0$. Then, for all $j\geq 0$ and $\mu^j=P_j(\mu)$,
$$\|\mu^j\|_{L^{p}_{\gamma}(\R^d)}\lesssim_{\varepsilon} 2^{j(\gamma+(d-s+\varepsilon)\frac{1}{p'})}.$$   
\end{lem}

Next, we have the following lemma.

\begin{lem}\label{lem proj of product}
     For all $1\leq p\leq \infty$ and $\mu$ as in Lemma \ref{IL lemma}, one has for $h(y)\in L^p(d\mu)$ that for all $j\geq 0$,
\begin{equation}
\|P_j(\mu h )\|_{L^{p}(dy)}\lesssim_\varepsilon 2^{j(d-s+\varepsilon)\frac{1}{p'}}\|h\|_{L^{p}(d\mu)}.
\end{equation}
\end{lem}

\begin{proof}
    We check the Lemma for $p=1$ and $p=\infty$ and then use interpolation.
For $p=1$, by recalling the definition of $P_j$ in Section \ref{sec: tools} we have
\begin{equation}
    \int |P_j(\mu h)(y)|dy\leq \int \int |2^{jd}\check{\beta}(2^j(y-x)) h(x)|d\mu(x) dy\leq \|\check{\beta}\|_{L^1(\R^d)}\int |h(x)|d\mu(x)\lesssim \|h\|_{L^{1}(\mu)}
\end{equation}
where in the second inequality, we used Fubini's Theorem to integrate first in $y$. For $p=\infty$, we have 
$$|P_j(\mu h)(y)|\leq 2^{jd}\int |\check{\beta}(2^{j}(y-x))||h(x)|d\mu(x).$$
From the fast decay of $\check{\beta}$ the main contribution comes from $|x-y|\leq 2^{-j}$ and
\begin{equation}
    \begin{split}
     2^{jd}\int_{x\colon |x-y|\leq 2^{-j} } |h(x)| d\mu(x)\lesssim 2^{jd}\mu(B(y,2^{-j}))\|h\|_{L^{\infty}(\mu)}\lesssim 2^{j(d-s+\varepsilon)}\|h\|_{L^{\infty}(\mu)}.
    \end{split}
\end{equation}
\end{proof}

As an immediate consequence, one has
\begin{equation}\label{factiks}
 \|P_j(g_t \mu)(y)\|_{L^{q}(dy dt)}\lesssim_{\varepsilon} 2^{j(d-s+\varepsilon)\frac{1}{q'}}\|g\|_{L^{q}(d\mu \times dt)}  
\end{equation}(also see \cite[page 78]{IKSTU}).

We remark that if $\mu$ is a compactly supported finite measure satisfying $\mu(B(x,r))\leq C r^{s_{\mu}}$ for all $r>0$ and $x\in \R^d$, then for any $\delta>0$, it is also true for some $C'>C$ that $\mu^{\delta}(B(x,r))\leq C'r^{s_{\mu}}$ where $C'$ is independent of $\delta>0$, so there is no problem in replacing $\mu$ by its mollifications in the lemmas above.

Going back to (\ref{computationthm1}), we will have
\begin{equation}
    \begin{split}
      \sum_{j\geq 0} \| \mu_E^j\|_{L^{p}_{-\eta+2\varepsilon}(\R^d)}\|\tilde{P}_j(g_t \mu_F)(y)\|_{L^{p'}(dy dt)}\lesssim \sum_{j\geq 0} 2^{j(-\eta+(d-s_E)\frac{1}{p'}+(d-s_F)\frac{1}{p}+3\varepsilon)}.
    \end{split}
\end{equation}

Finally, note that the series is summable for some small $\varepsilon>0$ since 
$d-\eta<\frac{1}{p'}s_E+\frac{1}{p} s_F$.

The reader should note that, although we gave our argument for a fixed Frostman probability measure $\mu_E$, for $A\subset E, \,\mu_E(A)>0$ if we replace $\mu_E$ with the restricted measure $\mu_A$ as in (\ref{def: restrictionmeasure}) then the argument still holds, with the implicit constant in the estimate \eqref{reduction2} merely being multiplied by a factor of $\mu_E(A)^{-1}$. This simple observation will be used in the proof of the result for trees later.

\section{Proof of Theorem \ref{Thm:pinnednonemptytreeinanydimensionwithlocalsmoothing}: Toy cases}\label{sec:toycases}

In this section, we explore a few toy cases of trees to illustrate the proof strategy of Theorem \ref{Thm:pinnednonemptytreeinanydimensionwithlocalsmoothing}. We start with chains, which are some of the simplest types of trees. Given a compact set $E\subset \R^d$, $d\geq 2$ and $k\geq 1$, let us consider the set of $k$ chains in $E$ given by 
$$C(E,k):=\{(|x_0-x_1|,|x_1-x_2|,\dots ,|x_{k-1}-x_k|)\colon x_0,x_1,\dots, x_k\text{ distinct points in } E\}$$
and its pinned (at the starting point) variant  
$$C^{x}(E,k):=\{(|x-x_1|,|x_1-x_2|,\dots ,|x_{k-1}-x_k|)\colon x_1,\dots, x_k \text{ distinct points in } E\}\subset \R^{k}_{+}.$$

Instead of restricting ourselves to pins at the starting point of the chain, we can also allow pins in more general locations. Let $0\leq p\leq k $ correspond to the location of the pin in the chain. 
\begin{equation}
    \begin{split}
    C^{x,p}(E,k):=\{(|x_0-x_1|,&|x_1-x_2|,\dots, |x_{p-1}-x|,|x-x_{p+1}|,\dots,|x_{k-1}-x_k|):\\
    & x_0,x_1, \dots, x_{p-1},x_{p+1}, \dots , x_k\text{ distinct points in }E\}.    
    \end{split}
\end{equation}

As we will see in the definition in the next section, $C^{x,p}$ are special types of pinned trees. In the first two toy cases below, we only explore the case when $p=0$, i.e. the endpoint pin case. In the proof of the general case of Theorem \ref{Thm:pinnednonemptytreeinanydimensionwithlocalsmoothing}, to be presented in the next section, one will see that the degree of the pinned vertex of the tree plays a key role.

\subsection{Toy case: 2 chains with a pin at an endpoint}\label{sec: 2chains}

We begin by working with a slightly smaller 2-chain set. Given $E\subset \mathbb{R}^d$ with ${\rm dim}(E)=:s_E$, let $E_0$, $E_1$, $E_2$ be pairwise separated subsets of $E$ with positive $(s_E-\varepsilon)$-dimensional Hausdorff measure. (Such a construction is standard. For example, it can be found in \cite[Lemma 3]{BIO23}.) Then there exist probability measures $\mu_0$, $\mu_1$, $\mu_2$ supported on $E_0$, $E_1$, $E_2$ respectively satisfying the Frostman condition (\ref{Frostman}). We will be studying the pinned chain set
\[
C^x(E_1, E_2, 2):=\{(|x-x_1|,|x_1-x_2|):\, x_1\in E_1,\, x_2\in E_2\},\quad x\in E_0.
\]

One can then construct a probability measure in $C^{x}(E_1,E_2,2)$ by taking the pushforward of the measure $\mu_1\times \mu_2$ by the map $\Phi^{x}:E_1\times E_2\rightarrow \R^2$, $\Phi^{x}(x_1,x_2)=(|x-x_1|,|x_1-x_2|)$ which will be given (actually for the mollifications of $\mu_1$, $\mu_2$) as

$$\tilde{\nu}^{x}(t_1,t_2)=(t_1t_2)^{d-1}A_{t_1}(\mu_1 A_{t_2}\mu_2)(x)$$
where 
$A_t(f)(x)$ is as in (\ref{def: averageAt}). Indeed, for any test function $\varphi\in C^{\infty}(\R^2)$, by using the change of variables $(y_1,y_2)=(x_1-x,x_2-x_1)$ followed by polar coordinates we get

\begin{equation}\label{density2chainpinnedatleaf}
    \begin{split}
        \int \varphi(t_1,t_2) (\tilde{\nu}^x)^{\mathbf{\delta}}(t_1,t_2)=&\int\int \varphi(|x_1-x|,|x_2-x_1|)\mu_1^{\delta_1}(x_1)\mu_2^{\delta_2}(x_2)\,dx_1dx_2\\
        =&\int \int \varphi(|y_1|,|y_2|)\mu_1^{\delta_1}(y_1+x)\mu_2^{\delta_2}(y_1+y_2+x)\,dy_1dy_2\\
        =&\int \int \varphi(t_1,t_2)(t_1t_2)^{d-1}\cdot\\
        &\qquad\int_{S^{d-1}}\int_{S^{d-1}}\mu_1^{\delta_1}(x+t_1\omega_1)\mu_{2}^{\delta_2}(x+t_1\omega_1+t_{2}\omega_2)
        \,d\sigma(\omega_2) d\sigma(\omega_1)dt_1dt_2\\
        =&\int\int \varphi(t_1,t_2)  (t_1t_2)^{d-1}\cdot\\
        &\qquad\int_{S^{d-1}}\mu_1^{\delta_1}(x+t_1\omega_1)A_{t_2}(\mu_2^{\delta_2})
        (x+t_1\omega_1) \,d\sigma(\omega_1) dt_1dt_2\\
        =&\int\int \varphi(t_1,t_2) (t_1t_2)^{d-1} A_{t_1}(\mu_1^{\delta_1} A_{t_2}(\mu_{2}^{\delta_2}))(x) \,dt_1dt_2.
    \end{split}
\end{equation}

Since the factor $t_1^{d-1}t_2^{d-1}$ is smooth in $[1,2]^{2}$ and is bounded from above and below, let us drop it as before and redefine 

$$\nu^{x}(t_1,t_2)=A_{t_1}(\mu_1 A_{t_2}\mu_2)(x)=\sigma_{t_1}*(\mu_1(\sigma_{t_2}*\mu_2))(x).$$The measure $\nu^{x}$ is supported in $C^{x}(E_1,E_2,2)$ and we would like to show that it is a continuous function of $(t_1,t_2)$ for $\mu_0$ a.e. $x\in E_0$, under certain dimensional assumption on the set $E$.

For fixed $x$, by Sobolev embedding in $\R$, we can bound for any $\gamma>1/p$ (here $p$ is the parameter in the assumed local smoothing estimate \eqref{partiallocalsmoothing}),

\begin{equation}\label{continuityofnux}
    \begin{split}
        &|\nu^x(t_1+h_1,t_2+h_2)-\nu^x(t_1,t_2)|\\
        \leq & |\nu^x(t_1+h_1,t_2+h_2)-\nu^x(t_1,t_2+h_2)|
        +|\nu^x(t_1,t_2+h_2)-\nu^x(t_1,t_2)|\\
        \lesssim &|h_1|^{\gamma-1/p} \|A_{s_1}(\mu_1 A_{t_2+h_2}\mu_2)(x)\|_{L^{p}_{\gamma}(ds_1)}
        +|h_2|^{\gamma-1/p} \|A_{t_1}(\mu_1 A_{s_2}\mu_2)(x)\|_{L^{p}_{\gamma}(ds_2)},\\
       \leq &{ |h_1|^{\gamma-1/p} \sup_{s_2\sim 1}\|\partial_{s_1}^{\gamma}A_{s_1}(\mu_1 A_{s_2}\mu_2)(x)\|_{L^{p}(ds_1)}
        +|h_2|^{\gamma-1/p} \sup_{s_1\sim 1}\|\partial_{s_2}^{\gamma}A_{s_1}(\mu_1 A_{s_2}\mu_2)(x)\|_{L^{p}(ds_2)},}\\
        \lesssim &{(|h_1|^{\gamma-\frac{1}{p}}+|h_2|^{\gamma-\frac{1}{p}})\|\partial_{s_1}^{\gamma}\partial_{s_2}^{\gamma}A_{s_1}(\mu_1 A_{s_2}\mu_2)(x)\|_{L^{p}(ds_1ds_2)}}
    \end{split}
\end{equation}
where from the second to third line in the equation above we are making use of Theorem \ref{tool: embeddingthm} to guarantee H\"older continuity for the functions 
$s\mapsto \nu^x(s,t_2+h_2)$ and $s\mapsto \nu^x(t_1,s)$.

Let us suppose temporarily that for any $t_1,t_2\in (1,2)$ the right-hand side is bounded by $C_x(|h_1|^{\gamma-1/p}+|h_2|^{\gamma-1/p})$ for $\mu_0$-a.e. $x$. Pick one such $x$, call it $x_0$, where this holds. Then we have that
\[
\frac{|\nu^{x_0}(t_1+h_1,t_2+h_2)-\nu^{x_0}(t_1,t_2)|}{|(h_1,h_2)|^{\gamma-1/p}}\lesssim \frac{|\nu^{x_0}(t_1+h_1,t_2+h_2)-\nu^{x_0}(t_1,t_2)|}{|h_1|^{\gamma-1/p}+|h_2|^{\gamma-1/p}}\le C_{x_0}
\]
for all sufficiently small $(h_1,h_2)$ and $(t_1,t_2)\in (1,2)^2$. Thus, the density $\nu^{x_0}$ is continuous for all $t_1,t_2\in(1,2)$ and consequently $C^{x_0}(E,2)$ has nonempty interior.

The following proposition will be enough to guarantee that the desired bound for the righthand side of (\ref{continuityofnux}) for $\mu_0$ a.e. $x$.

\begin{prop}\label{2chainsprop}
   Let $d\geq 2$. Assume that the local smoothing estimate \eqref{partiallocalsmoothing} holds for some $p\in (2,\infty)$ and $\eta>0$, and that $E$ is a Borel subset of $\R^d$ with $\Hdim(E)>d-\eta$. Let $\mu_i$, $i=0,1,2$, be Frostman probability measures on $E_i$, $i=0,1,2$, constructed as in the above, satisfying the condition (\ref{Frostman}) with $\varepsilon$ sufficiently small. Let $\nu^x(t_1,t_2)=A_{t_1}(\mu_1 A_{t_2} \mu_2)(x)$ and $\gamma=1/p+\varepsilon$, then
 $$\left(\int_{\R^d} \int_{[1,2]^2}|\partial_{t_1}^{\gamma}\partial_{t_2}^{\gamma}\nu^x(t_1,t_2)|^pdt_1dt_2 \,d\mu_0(x)\right)^{1/p}\leq C_{\varepsilon} $$
 for some $C_{\varepsilon}>0$.  
\end{prop}

\begin{proof}[Proof of Proposition \ref{2chainsprop}]
    
    As before, we will be working with mollified versions of the measures $\mu_0, \mu_1,\mu_2$. We still use the same notation for them for the sake of simplicity. 

 By duality, 
    \[
   \begin{split}
\|\partial_{t_1}^{\gamma}&\partial_{t_2}^{\gamma}\nu^x(t_1,t_2)\|_{L^p(d\mu_0 dt_1dt_2)}\\=&\sup_{\|g\|_{L^{p'}(d\mu_0 dt_1dt_2)}=1} \left|\int_{1}^{2} \int_{1}^{2}\int_{\R^d} \partial_{t_1}^{\gamma}\partial_{t_2}^{\gamma}\nu^x(t_1,t_2)g(x,t_1,t_2)\,d\mu_0(x)dt_1dt_2\right|.
    \end{split}
    \]

  Fix \(g\in L^{p'}(d\mu_0\,dt_1\,dt_2)\) with
\[
\|g\|_{L^{p'}(d\mu_0\,dt_1\,dt_2)}=1,
\] and write \(g_{t_1,t_2}(x):=g(x,t_1,t_2)\). Decomposing in the \(x\)-variable we get:
\begin{equation}\label{eq:2chaindualityargument}
\begin{aligned}
&\left|
\int
\partial_{t_1}^{\gamma}
A_{t_1}\bigl(\mu_1\,\partial_{t_2}^{\gamma}A_{t_2}\mu_2\bigr)(x)
g(x,t_1,t_2)\,d\mu_0(x)\,dt_1\,dt_2
\right|                                      \\
&\qquad\lesssim
\sum_{j\geq 0}
\left\|
\partial_{t_1}^{\gamma}
A_{t_1}
P_j\bigl(\mu_1\,\partial_{t_2}^{\gamma}A_{t_2}\mu_2\bigr)
\right\|_{L^p(dx\,dt_1\,dt_2)}
\left\|
\widetilde P_j(g_{t_1,t_2}\mu_0)
\right\|_{L^{p'}(dx\,dt_1\,dt_2)} .
\end{aligned}
\end{equation}
    
To estimate the first factor, we note that by using the local smoothing hypothesis in the variables $(x,t_1)$ with $t_2$ fixed and then taking $L^{p}([1,2])$ norm in the variable $t_2$ we get
\[
\left\|
\partial_{t_1}^{\gamma}
A_{t_1}
P_j\bigl(\mu_1\,\partial_{t_2}^{\gamma}A_{t_2}\mu_2\bigr)
\right\|_{L^p(dx\,dt_1\,dt_2)}
\lesssim
2^{j(-\eta+2\varepsilon)}
\left\|
P_j\bigl(\mu_1\,\partial_{t_2}^{\gamma}A_{t_2}\mu_2\bigr)
\right\|_{L^p(dx\,dt_2)} .
\]

By Lemma \ref{lem proj of product} applied to the measure \(\mu_1\) and integrated in $t_2$,
\[
\left\|
P_j\bigl(\mu_1\,\partial_{t_2}^{\gamma}A_{t_2}\mu_2\bigr)
\right\|_{L^p(dx\,dt_2)}
\lesssim
2^{j(d-s_E+\varepsilon)/p'}
\left\|
\partial_{t_2}^{\gamma}A_{t_2}\mu_2
\right\|_{L^p(d\mu_1\,dt_2)} .
\]

From the quantitative estimate we proved to get pinned nonempty interior for Falconer pinned distance problem, i.e., estimate (\ref{thickened2}) in the proof of Theorem \ref{Thm:pinnednonemptyassumingLS}, we know that since $s_{E}>d-\eta$, 
$$\|\partial_{t_{2}}^{1/p+\varepsilon}A_{t_2} \mu_2\|_{L^{p}(d\mu_1 dt_2)}\lesssim_{\varepsilon} 1, $$
for each $\varepsilon$ sufficiently small. Hence
\[
\left\|
\partial_{t_1}^{\gamma}
A_{t_1}
P_j\bigl(\mu_1\,\partial_{t_2}^{\gamma}A_{t_2}\mu_2\bigr)
\right\|_{L^p(dx\,dt_1\,dt_2)}
\lesssim
2^{j(-\eta+2\varepsilon)}
2^{j(d-s_E+\varepsilon)/p'} .
\]
For the second factor in the right-hand side of (\ref{eq:2chaindualityargument}), we use that as consequence of Lemma \ref{lem proj of product}, 
\[
\left\|
\widetilde P_j(g_{t_1,t_2}\mu_0)
\right\|_{L^{p'}(dx\,dt_1\,dt_2)}
\lesssim
2^{j(d-s_E+\varepsilon)/p}
\|g\|_{L^{p'}(d\mu_0\,dt_1\,dt_2)} .
\]
Therefore the whole sum is bounded by
\[
\sum_{j\geq 0}
2^{j(-\eta+2\varepsilon)}
2^{j(d-s_E+\varepsilon)/p'}
2^{j(d-s_E+\varepsilon)/p}
=
\sum_{j\geq 0}
2^{j(-\eta+d-s_E+3\varepsilon)} .
\]
This series converges after choosing \(\varepsilon>0\) sufficiently small, because
\(s_E>d-\eta\).

    \end{proof}

\subsection{Another toy case: 3-chains}
In this section, we will give a sketch of the proof of the nonempty interior result for 3-chains which are pinned at a leaf (endpoint). This is a prototypical case that illustrates many of the techniques that we will need when dealing with the most general case of arbitrary trees. We will focus on showcasing the inductive scheme, i.e. on how one can apply the 2-chain result proved in the previous subsection to the 3-chain case.

Similarly to the 2-chain case, one starts with finding four pairwise separated subsets $E_0,E_1,E_2,E_3$ in $E$ and constructing Frostman measures $\mu_0,\mu_1,\mu_2,\mu_3$ on them. After dropping some $t_i$ dependent terms (smooth and comparable to 1), one wants to study the continuity of
\[
A^{(3)}_{t_1,t_2,t_3}(\mu_1,\mu_2,\mu_3)(x):=A_{t_1}(\mu_1 A_{t_2}(\mu_2A_{t_3}\mu_3))(x),\quad x\in E_0.
\]For a fixed $x$, the same Sobolev embedding argument reduces the matter to estimating the following:
\begin{equation}\label{eqn: 3-chain}
\|\partial^\gamma_{t_1}\partial^\gamma_{t_2}\partial^\gamma_{t_3}A^{(3)}_{t_1,t_2,t_3}(\mu_1,\mu_2,\mu_3)(x)\|_{L^p(d\mu_0dt_1dt_2dt_3)}<\infty,
\end{equation}
where $\gamma=\frac{1}{p}+\varepsilon$ (for some $\varepsilon$ depending on $d,s_E,\eta$), where as a reminder $dt_i$ stands for the Lebesgue measure in the interval $[1,2]$.

By duality, for some $g(y,t_1,t_2,t_3)=g_{\mathbf{t}}(y)$ with $\|g\|_{L^{p'}(d\mu_0d\mathbf{t})}=1$, $\mathbf{t}:=(t_1,t_2,t_3)$, one can bound
\[
\begin{split}
&\|\partial^\gamma_{t_1}\partial^\gamma_{t_2}\partial^\gamma_{t_3}A^{(3)}_{t_1,t_2,t_3}(\mu_1,\mu_2,\mu_3)(x)\|_{L^p(d\mu_0d\mathbf{t})}\\
\leq & \sum_{j\geq 0}\left|\int_{[1,2]^3} \int_{\R^d} \partial^\gamma_{t_1}A_{t_1}(P_j(\mu_1\partial_{t_2}^{\gamma}\partial_{t_3}^{\gamma}A^{(2)}_{t_2,t_3}(\mu_2,\mu_3)))(y)\widetilde{P}_j(g_{\mathbf{t}}(y)\mu_0(y))\,dyd\mathbf{t}\right|\\
\lesssim & \sum_{j\geq 0}\left\| \partial^\gamma_{t_1}A_{t_1}(P_j(\mu_1\partial_{t_2}^{\gamma}\partial_{t_3}^{\gamma}A^{(2)}_{t_2,t_3}(\mu_2,\mu_3)))(y)\right\|_{L^p(dyd\mathbf{t})}\left\| \widetilde{P}_j[g_{\mathbf{t}}(y)\mu_0(y)]\right\|_{L^{p'}(dyd\mathbf{t})}. 
\end{split}
\]Here, we have adopted the notation $A^{(2)}_{t,s}(\mu,\nu):=A_{t}(\mu A_s(\nu))$. The second factor can be estimated directly using Lemma \ref{lem proj of product} for each fixed $\mathbf{t}$ and then integrating in $\mathbf{t}$. To handle the first factor, we use the 2-chain result and a combination of techniques that have been used in the previous subsection. More precisely, by local smoothing and Lemma \ref{lem proj of product}, one has
\begin{align*}
   \| \Vert \partial^{\gamma}_{t_1}A_{t_1}&(P_j(\mu_1 \partial_{t_2}^{\gamma}\partial_{t_3}^{\gamma}A_{t_2,t_3}^{(2)}(\mu_2,\mu_3)))\Vert_{L^p(dy\,dt_1)}\|_{L^p(dt_2dt_3)} \\
   \lesssim &2^{j(-\eta+2\varepsilon)}\|\Vert P_j(\mu_1 \partial_{t_2}^{\gamma}\partial_{t_3}^{\gamma} A_{t_2,t_3}^{(2)}(\mu_2,\mu_3))(y)\Vert_{L^p(dy)}\|_{L^p(dt_2dt_3)} \\
    \lesssim &2^{j(-\eta+2\varepsilon)}2^{\frac{j}{p'}(d-s_E+\varepsilon)}\Vert\partial_{t_2}^{\gamma}\partial_{t_3}^{\gamma} A_{t_2,t_3}^{(2)}(\mu_2,\mu_3)\Vert_{L^p(d\mu_1dt_2 dt_3)}.
    \end{align*}
    The result for 2-chains in Proposition \ref{2chainsprop} implies that
    \[
    \Vert \partial_{t_2}^{1/p+\varepsilon}\partial_{t_3}^{1/p+\varepsilon}A^{(2)}_{t_2,t_3}(\mu_2,\mu_3) \Vert_{L^p(d\mu_1\,dt_2\,dt_3)}\le C.
    \]
Since $s_E>d-\eta$, we can again choose $\epsilon>0$ sufficiently small so that the series in $j\geq 0$ converges.

\subsection{The final toy case: star graphs}

We will say a graph $G$ is a star graph if it is a tree and there exists a distinguished vertex $v$ such that every edge in $G$ has $v$ as one of its endpoints. Let $G_k$ denote a tree with $k+1$ vertices $v_0,\ldots,v_k$ where $v_0$ is the distinguished vertex (and also the point at which we pin).

We start by taking $E_0,E_1,\dots ,E_k$, a collection of $(k+1)$ pairwise separated subsets of $E$ with positive $(s_E-\varepsilon)$-dimensional Hausdorff measure. Construct for each $i$, $i=0,\cdots,k$, a probability measure $\mu_i$ supported on $E_i$, satisfying the Frostman condition \eqref{Frostman}.

Fix $x\in E_0$ and let $t_j=|v_0-v_j|$. Then the relevant density to consider (from Lemma \ref{lem split of densities} in the next section) is
\[
\nu^x(t_1,\ldots,t_k)=\prod_{i=1}^kA_{t_i}(\mu_i)(x).
\]
When $\dim_{\mathcal{H}}(E)>d-\eta$ (where $\eta$ is the parameter in the local smoothing estimate \eqref{partiallocalsmoothing}), this is easily seen to be continuous for $\mu_0$ a.e. $x$ since from the proof of Theorem \ref{Thm:pinnednonemptyassumingLS} that would be a product of continuous densities. However, our general induction argument for arbitrary trees will require us to prove the stronger quantitative estimate that there exists a constant $C_k$ (which may also depend on $\gamma=\frac{1}{p}+\varepsilon$ and $\mu_i$) such that for $\bm{t}=(t_1,t_2,\dots, t_k)$
\[
\left\Vert\prod_{i=1}^k\partial_{t_i}^{\gamma}A_{t_i}(\mu_i)(x)\right\Vert_{L^p(d\tilde{\mu}_0(x)\,d\bm{t})}\le C_k,
\]
for some $\gamma=\frac{1}{p}+\varepsilon$, and some measure $\tilde{\mu}_0$ on $E_0$ that is the normalized restriction of $\mu_0$ on a subset of positive $\mu_0$ measure. To construct $\tilde{\mu}_0$, we first recall from \eqref{reduction2} in the proof of Theorem \ref{Thm:pinnednonemptyassumingLS} that for each $i$, $$h_i(x):=\|\partial_{t_i}^{\gamma}A_{t_i}(\mu_i)(x)\|_{L^p(dt_i)}\in L^p(d\mu_0(x)).$$ Let $h(x):=\sum_{i=1}^k h_i(x)$. Note that $h\in L^p(d\mu_0)$ since it is a finite sum of functions in $L^p(\mu_0)$. So we can find a subset of $\tilde{E}_0\subset E_0$ satisfying $\mu_0(\tilde{E}_0)\ge 1-2^{-k}$ and $|h(x)|\le N_k$ for all $x\in \tilde{E}_0$ and some sufficiently large $N_k$. We finally define the measure
\[
\tilde{\mu}_0(S)=\frac{\mu_0(S\cap \tilde{E}_0)}{\mu_0(\tilde{E}_0)}.
\]

Returning to the estimate at hand, we have
\[
\left\Vert\prod_{i=1}^k \partial_{t_i}^{\gamma}A_{t_i}(\mu_i)\right\Vert_{L^p(d\tilde{\mu}_0\,d\bm{t})}\le N_k^{k-1} \Vert\partial_{t_1}^{\gamma}A_{t_1}(\mu_1)\Vert_{L^p(d\tilde{\mu}_0\,dt_1)}\le C_k,
\]
giving the desired estimate.

\section{Proof of Theorem \ref{Thm:pinnednonemptytreeinanydimensionwithlocalsmoothing}: General trees}\label{sec trees}

\subsection{Trees realized in a compact set}
In this section, we use similar notation to that used in \cite{OuTaylor} with a few changes.

A tree is a graph in which each pair of vertices are connected by exactly one path. Let $\mathcal{T}=(\mathcal{V},\mathcal{E})$ be a tree where $\mathcal{V}=\{v_0,v_1,\dots, v_n\}$ is the set of vertices of $\mathcal{T}$ and $\mathcal{E}\subset \{(v_i,v_j)\colon v_i,v_j\in \mathcal{V}\text{ and } i<j\}$ is the list of all the edges of the tree $\mathcal{T}$. It is not hard to check that if $\mathcal{T}$ is a tree with $n+1$ vertices then $\mathcal{T}$ has $n$ edges, that is, $|\mathcal{E}|=n$. In that case, we can refer to $\mathcal{T}$ as an $n$-tree of shape $\mathcal{E}$, and when it is convenient to recall that $\mathcal{T}$ has $n$ edges we may denote it $\mathcal{T}^{n}$.

Recall from \cite{OuTaylor} that a natural way to enumerate the set $\mathcal{E}$ is given by 
\begin{equation}\label{edgeenumeration}
    \mathcal{E}=\{(v_{i_1},v_{i_2}),\,(v_{i_3},v_{i_4}),\,\dots ,\,(v_{i_{2n-1}},v_{i_{2n}})\}
\end{equation}
where $0\leq i_1,i_2,\dots ,i_{2n}\leq n$ and the following conditions are satisfied:
$i_1\leq i_3\leq \dots \leq i_{2n-1} $ and $i_{2s}<i_{2t}$ whenever $s<t$ and $i_{2s-1}=i_{2t-1}$. 

\begin{defn}
 Let $\mathcal{T}=(\mathcal{V},\mathcal{E})$ be a tree and $E\subset \R^d$ a compact set. The \textit{edge-length set of $n$-trees of shape $\mathcal{E}$ (or $\mathcal{T})$ generated by $E$} can be defined as 
$$\Delta_{\mathcal{T}}(E)=\{(|x_i-x_j|)_{ (v_i,v_j)\in \mathcal{E} }\colon  x_0,x_1,x_2,\dots ,x_n \,\text{ are distinct points in }E\}\subset \R_{+}^{| \mathcal{E}|}=\R_{+}^n.$$
\end{defn}

When enumerating the coordinates of the vector $(|x_i-x_j|)_{ (v_i,v_j)\in \mathcal{E} }$, we implicitly assume we are using the enumeration of $\mathcal{E}$ described before in (\ref{edgeenumeration}).

We will also be interested in pinned versions of $\Delta_{\mathcal{T}}(E)$. Suppose by relabeling the vertices of $\mathcal{T}$ if necessary, that $v_0$ is the location in the tree $\mathcal{T}$ where we want to have the pin. We then define for $x_0\in E$ 
\begin{equation}\label{def pin tree}
    \Delta_{\mathcal{T},v_0}^{x_0}(E):=\{(|x_i-x_j|)_{ (v_i,v_j)\in \mathcal{E} }\colon x_1,x_2,\dots ,x_n \,\text{ are distinct points in }E\}\subset \R^{n}_{+}
\end{equation}
which represents the \textit{edge-length set of the tree $\mathcal{T}$ generated by $E$ with $v_0$-pin at $x_0\in E$}.

\begin{ex}\label{example tree}

Consider the tree $\mathcal{T}=(\mathcal{V},\mathcal{E})$, with $\mathcal{V}=\{v_0,v_1,v_2,v_3,v_4\}$ and  $$\mathcal{E}=\{(v_0,v_1),(v_0,v_2),(v_0,v_4),(v_2,v_3)\}$$ illustrated below. When we pin the tree at $v_0$ and consider the edge-lengths of $\mathcal{T}$-trees in $E$ pinned at a point $x_0\in E$ we get $$\Delta_{\mathcal{T},v_0}^{x_0}(E)=\{(|x_0-x_1|,|x_0-x_2|,|x_0-x_4|,|x_2-x_3|)\colon x_1,x_2,x_3,x_4\in E\}.$$

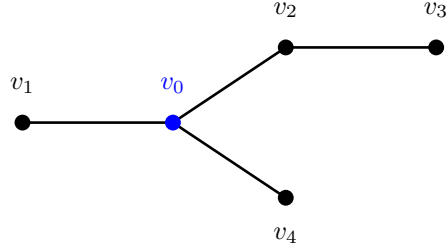
\begin{figure}[h]\label{figure1}
\begin{center}
         \scalebox{1}{
\begin{tikzpicture}

\fill (0,0) circle (3pt);

\fill (3.5,1) circle (3pt);
\fill (3.5,-1) circle (3pt);
\fill (5.5,1) circle (3pt);

\draw[line width=1pt] (0,0)--(2,0)--(3.5,1)--(5.5,1);
\draw[line width=1pt] (2,0)--(3.5,-1);

\draw (0,0.5) node {$v_1$};
\draw[blue] (2,0.5) node {$v_0$};
\draw (3.5,1.5) node {$v_2$};
\draw (3.5,-1.5) node {$v_4$};
\draw (5.5,1.5) node {$v_3$};

\fill[blue] (2,0) circle (3pt);
\end{tikzpicture}}
\end{center}
\caption{A sample tree, pinned at $v_0$.}
\end{figure}
    
\end{ex}
\subsection{Pinned result for trees}
Let $\mathcal{T}=(\mathcal{V},\mathcal{E})$ be a tree graph and let $v_0$ be a distinguished vertex in $\mathcal{T}$. Denote by $m\geq 1$ the multiplicity of the vertex $v_0$. The subgraph induced by $\mathcal{V}\setminus\{v_0\}$ has $m$ many connected components, which we will enumerate and call $H_1'=(\mathcal{V}_1',\mathcal{E}_1'),\ldots, H_m'=(\mathcal{V}_m',\mathcal{E}_m')$, where $m=1$ if and only if $v_0$ is a leaf of the tree. For each $l=1,\ldots,m$, let $v_{l,1}\in H_l'$ be the unique vertex such that $(v_{l,1},v_0)\in \mathcal{E}$ (and we will write all its other vertices of $H_{l}'$ in the form $v_{l,j}$). Finally let $H_l=(\mathcal{V}_l'\cup\{v_0\}, \mathcal{E}_l'\cup\{(v_0,v_{l,1})\})=:(\mathcal{V}_l,\mathcal{E}_l)$, i.e. $H_l$ is the subgraph of $\mathcal{T}$ induced by $H_l'\cup\{v_0\}$. See figure \ref{fig tree decomp} for an illustration of such decomposition for the tree in Example \ref{example tree}.

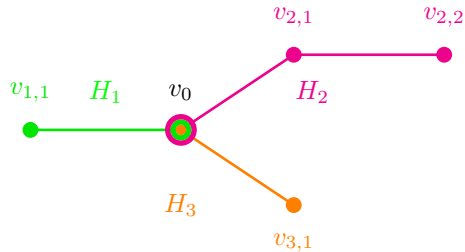
\begin{figure}[h]\label{fig tree decomp}
\begin{center}
         \scalebox{1}{
\begin{tikzpicture}
\draw[line width=1pt, color=black!10!green] (0,0)--(2,0);
\draw[line width=1pt, magenta] (2,0)--(3.5,1)--(5.5,1);
\draw[line width=1pt,orange] (2,0)--(3.5,-1);

\draw[black!10!green] (0,0.5) node {$v_{1,1}$};
\draw (2,0.5) node {$v_0$};
\draw[magenta] (3.5,1.5) node {$v_{2,1}$};
\draw[orange] (3.5,-1.5) node {$v_{3,1}$};
\draw[magenta] (5.5,1.5) node {$v_{2,2}$};
\fill[black!10!green] (0,0) circle (3pt);
\fill[magenta] (3.5,1) circle (3pt);
\fill[orange] (3.5,-1) circle (3pt);
\fill[magenta] (5.5,1) circle (3pt);
\fill[orange] (2,0) circle (2pt);
\draw[color=black!10!green,line width=2pt] (2,0) circle (3pt);
\draw[magenta,line width=2pt] (2,0) circle (5pt);

\draw[magenta] (3.75,0.5) node {$H_2$};
\draw[green](1,0.5) node {$H_1$};
\draw[orange](2,-1) node {$H_3$};

\end{tikzpicture}}
\end{center}
\caption{Decomposition in subtrees for tree in Example \ref{example tree}.}
\end{figure}

Fix a compact set $E\subset \R^d$ and $x_0\in E$. For our applications, we will want to divide $E$ into separated sets $E_1,E_2,\ldots,E_k$ as we did in Section \ref{sec:toycases}, but we will abuse notation by calling all these sets $E$. We want to first define graph maps which generalize the $\Phi^x$ maps from before. We recall the graph map $\Phi_{C_k}^{x_0}:\prod_{j=1}^k E\rightarrow\R^k$ for a $k$-chain $C_k$ pinned at a leaf $x_0$:
\[
\Phi_{C_k}^{x_0}(x_1,\ldots,x_k):=(|x_0-x_1|,|x_1-x_2|,\ldots,|x_{k-1}-x_k|).
\]

Let us define the graph map for a general $n$-tree $\mathcal{T}^n=(\mathcal{V}^{n+1},\mathcal{E})$ pinned at $v_0$ by using induction on the number of edges $n$. For $n=1$, the map $\Phi^{x_0}_{\mathcal{T}^1}:E\rightarrow \R_{+}$ is simply  $\Phi^{x_0}_{\mathcal{T}^{1}}(x_1)=|x_0-x_1|,\,x_1\in E$. Let $n\geq 2$. Assume we have it defined the graph map for any $j$-tree $\mathcal{T}^{j}$ with at most $j\leq n-1$ edges.

Consider an $n$-tree $\mathcal{T}^{n}=(\mathcal{V}^{n+1},\mathcal{E}^n)$ with distinguished vertex $v_0$ as above and let $m$ and the $H_l$'s be as above as well. Note that $|\mathcal{V}|=n+1$ and denoting $n_l=|\mathcal{V}_l'|$ one has that $\sum_{l=1}^m n_l=n$.

If $m\ge 2$ then we define its graph map $\Phi^{x_0}_{\mathcal{T}^n,v_0}:\prod_{i=1}^n E\rightarrow \R^{n}_{+}$ as
\begin{align*}
\Phi_{\mathcal{T}^{n},v_0}^{x_0}&(x_{1,1},x_{1,2},\ldots,x_{1,n_1},x_{2,1},\ldots,x_{2,n_2},\ldots,x_{m,1},\ldots,x_{m,n_m})\\
=&\bigg(\Phi_{H_1, {v_0}}^{x_{0}}(x_{1,1},\ldots,x_{1,n_1}),
 \Phi_{H_2,{v_0}}^{x_0}(x_{2,1},\ldots,x_{2,n_2}),\ldots,\Phi_{H_m,{v_0}}^{x_0}(x_{m,1},\ldots,x_{m,n_m})\bigg)
\end{align*}

If $m=1$, then we let $H'$ be the subgraph induced by deleting $v_0$ and let $v_1$ be the unique vertex in $\mathcal{T}$ such that $(v_0,v_1)$ is an edge. In this case, we define its graph map as
\[
\Phi_{\mathcal{T}^{n},v_0}^{x_0}(x_1,x_2,\ldots,x_n)=(|x_0-x_1|,\Phi_{H',{v_1}}^{x_1}(x_2,\ldots,x_n))
\]

That finishes the inductive definition of the graph map. 

Let $\nu_{\mathcal{T},v_0}^{x_0}=(\Phi_{\mathcal{T},v_0}^{x_0})_{*}(\prod_{j=1}^{n} \mu)$ be the pushforward of the measure $\prod_{j=1}^{n} \mu$ by the map $\Phi^{x_0}_{\mathcal{T},v_0}$. Note that $\nu^{x_0}_{\mathcal{T},v_0}$ is a probability measure supported in $\Delta_{\mathcal{T}^n,v_0}^{x_0}(E)$.

We want to express the density $\nu_{\mathcal{T},v_0}^{x_0}$ in terms of the densities $\nu_{H_l,v_{0}}^{x_0}$ (case $m\geq 2$) and terms of $\nu_{H',v_1}^{x_1}$ (case $m=1$), which is accomplished by the following lemma.

\begin{lem}\label{lem split of densities}
    Denote $\mathbf{t}=(\mathbf{t}_1,\mathbf{t}_2,\dots ,\mathbf{t}_m)$ and $\mathbf{t}_l=(t_{l,1},\dots ,t_{l,n_l})$. If $m\geq 2$ then one has
$$\nu^{x_0}_{\mathcal{T},v_0}(\mathbf{t})=\prod_{l=1}^{m} \nu^{x_0}_{H_l,v_0}(\mathbf{t_l}).$$
If $m=1$, then 
$$\nu_{\mathcal{T},{v_0}}^{x_0}(t_1,\dots ,t_n)=A_{t_1}(\mu(\cdot) \nu_{H',v_1}^{(\cdot)}(t_2,t_3,\dots, t_n))(x_0).$$
\end{lem}

\begin{proof}[Proof of Lemma \ref{lem split of densities}]

Let $\phi$ be a smooth test function. We recall that each measure $\mu$ has been mollified, though we suppress this from the notation. Then we have
\begin{align}
    F(\phi):=&\int \phi(t_1,\ldots,t_n)\nu^{x_0}_{\mathcal{T},v_0}(t_1,\ldots,t_n) dt_1\dots dt_n\\
    = &\int \phi(\Phi_{\mathcal{T},v_0}^{x_0}(x_{1,1},\ldots,x_{m,n_m}))\prod_{l=1}^m\prod_{j=1}^{n_l}\mu(x_{l,j})\,d\mathbf{x}\\
=& \int \phi(\Phi_{\mathcal{T},v_0}^{x_0}(\mathbf{x}_{1},\ldots,\mathbf{x}_{m}))\prod_{l=1}^m(\prod_{j=1}^{n_l}\mu)(\mathbf{x}_{l})\,d\mathbf{x}
\end{align}
where $\mathbf{x}_l=(x_{l,1},x_{l,2},\dots ,x_{l,n_l})$.
By construction every one of the $n$ components of the mapping $\Phi_{\mathcal{T},v_0}^{x_0}$ is of the form $|x_0-x_{l,1}|$ or $|x_{l,j}-x_{l,k}|$ where $(v_{l,j},v_{l,k})\in \mathcal{E}_l'$. We change variables $y_{l,1}=x_0-x_{l,1}$ and $y_{l,j,k}=x_{l,j}-x_{l,k}$ whenever $(x_{l,j},x_{l,k})\in \mathcal{E}_l'$. For each $l$, since $|\mathcal{E}'_l|=n_l-1$ we can relabel the variables $\{y_{l,j,k}\colon (v_{l,j},v_{l,k})\in \mathcal{E}_{l}' \}$ as a $\{y_{l,p}\colon 2\leq p\leq n_l\}$ and let $(y_{l,1}\cdots y_{l,n_l})=\Psi_l(x_{l,1},\cdots , x_{l,n_l})$, or in shorter notation $\mathbf{y}_l=\Psi_l(\mathbf{x}_l)$. The coordinate change map is then
$$\mathbf{y}=(\mathbf{y}_1,\dots,\mathbf{y}_m)=(\Psi_1(\mathbf{x}_1),\dots, \Psi_m(\mathbf{x}_m))=:\Psi(\mathbf{x}).$$

Additionally, notice that we have
\[
\Phi_{\mathcal{T},v_0}^{x_0}(\mathbf{x}_{1},\ldots,\mathbf{x}_{m})=\Big(\Phi_{H_1,v_0}^{x_0}(\mathbf{x}_{1}),\Phi_{H_2,v_0}^{x_0}(\mathbf{x}_{2}),\ldots,\Phi_{H_m,v_0}^{x_0}(\mathbf{x}_{m})\Big)
\]
by construction. In the new coordinate system, the argument of $\phi$ is very simple:
\begin{equation}
\Phi_{\mathcal{T},v_0}^{x_0}(\Psi^{-1}(\mathbf{y}))=(|\mathbf{y_1}|,|\mathbf{y_2}|,\ldots,|\mathbf{y_m}|),
\end{equation}
where we are abusing notation to mean $|\mathbf{y_l}|=(|y_{l,1}|,\ldots,|y_{l,n_l}|).$  In fact, we can even notice that the affine map $\Psi^{-1}$ can be represented with its linear part in a block matrix form, where the $l$th block is exactly the linear part of the map $\Psi_l^{-1}$ arising as a coordinate change $\mathbf{x}_l=\Psi_l^{-1}(\mathbf{y}_l)$ when computing $\nu_{H_l,v_0}^{x_0}$; we leave this as an exercise for the reader to verify. 

The Jacobian determinant of the change of variables splits $$J(\mathbf{y}):=|\det(\Psi^{-1})(\mathbf{y})|=\prod_{l=1}^{m} |\det(\Psi_l^{-1})(\mathbf{y}_l)|=:\prod_{l=1}^{m}J_l(\mathbf{y}_l).$$ Moreover each $\Psi_l^{-1}$ is affine with a linear part being a triangular matrix with 1's in the diagonal (so determinant 1), which gives $J_l\equiv1$. So after the change of variables one has

$$ F(\phi)=\int \phi(|\mathbf{y}_1|,\dots ,|\mathbf{y}_m|)\prod_{l=1}^m\left[(\prod_{j=1}^{n_l}\mu)(\Psi_l^{-1}(\mathbf{y}_l))\,\right]d\mathbf{y}$$

We then change each of the new variables $y_{l,j}$ in polar coordinates $y_{l,j}=t_{l,j}\omega_{l,j}$ for $1\leq j\leq n_l$, where the $\omega_{l,j}$'s are unit vectors. Thus, we can split the spherical integrals into an $m$-fold product, i.e.

$$F(\phi)=\int_{\mathbf{t}}\phi(\mathbf{t})\prod_{l=1}^m (t_{l,1}\dots t_{l,n_l})^{d-1} \left(\int_{(S^{d-1})^{n_l}} (\prod_{j=1}^{n_l} \mu)(\Psi_l^{-1}(t_{l,1}\omega_{l,1},\dots t_{l,n_l}\omega_{l,n_l}))\,d\sigma^{n_l}(\mathbf{\omega_l})\right)\,d\mathbf{t}.$$

Notice that because of the block structure of the coordinate change, the $l$th term in the product above is the density corresponding to the subtree $H_l$
since that would be the expression obtained for $\nu_{\mathcal{T},v_0}^{x_0}$ if this procedure was performed for $H_l$ instead of $\mathcal{T}$, so we can say that for each $l=1,\ldots,m$
\[
(t_{l,1}\dots t_{l,n_l})^{d-1}\int_{(S^{d-1})^{n_l}}\prod_{j=1}^{n_l} \mu(\Psi_l^{-1}(t_{l,1}\omega_{l,1},\dots t_{l,n_l}\omega_{l,n_l}))\,d\sigma^{n_l}(\mathbf{\omega_l})=\nu_{H_l,v_0}^{x_0}(\mathbf{t}_l)
\]
We will also alternatively denote $\nu_{H_l,v_0}^{x_0}(\mathbf{t}_l)=:A^{(H_l,v_0)}_{t_{l,1},\ldots,t_{l,n_l}}(\mu,\ldots,\mu)(x_0)$.
Thus, we have that
\[
\nu_{\mathcal{T},v_0}^x(\mathbf{t})=\prod_{l=1}^m \nu_{H_l,v_0}^x(\mathbf{t_l})=\prod_{l=1}^m A^{(H_l,v_0)}_{t_{l,1},\ldots,t_{l,n_l}}(\mu,\ldots,\mu)(x).
\]

We now turn to the case where $m=1$. Consider the tree $H'$ obtained by deleting $v_0$ from $H:=\mathcal{T}$. Put an orientation on the edges of $H'$ via the rule that an edge's direction goes from $v_i$ to $v_j$ if and only if $v_j$ is farther from $v_0$ than $v_i$ is (where distance is measured in the sense of the walk with the fewest edges). Form a matrix $M_{H'}$ whose $(i,j)$th entry is
\[
(M_{H'})_{i,j}=\delta_{ij}-\text{Edge}(j,i),
\]
where $\text{Edge}(j,i)$ is $1$ if there is a directed edge from $v_j$ to $v_i$ and $0$ otherwise. By relabeling the vertices if needed, we may assume $M_{H'}$ is a lower triangular matrix with 1s along the diagonal, one entry with value $-1$ in each row (except the first row), and all other entries are 0. If we let $\text{Abs}$ denote the map which replaces each entry of a vector with its absolute value, then by construction, $M_{H'}$ induces the graph map in the sense that
\[
\Phi_{H',v_1}^{x_1}(x_2,\ldots,x_n)=\text{Abs}(M_{H'}(x_2,\ldots,x_n)-\sum_{j\in\mathcal{J}}x_1e_j),
\]
where $\mathcal{J}$ is the set of all indices $j$ of vertices for which there is an edge from $v_1$ to $v_j$. Similarly, we have that
\[
\Phi_{H,v_0}^{x_0}(x_1,\ldots,x_n)=\text{Abs}(M_{H}(x_1,\ldots,x_n)-x_0e_1).
\]
By construction (and using the fact that $v_0$ is a leaf), we have that $M_H$ takes the form
\[
\begin{bmatrix}
    1 & 0 & \cdots & 0 \\
    -1 & M_{H',(1,1)} & \cdots & M_{H',(1,n-1)} \\
    \vdots & \vdots & \ddots & \vdots \\
    0 & M_{H',(n-1,1)} & \cdots & M_{H',(n-1,n-1)}
\end{bmatrix}
\]
We consider how to invert this matrix using row operations. First, the first row can be added to the second row (and no other rows, since $v_0$ is a leaf). Then, the second row can be added to all rows below it which have a $-1$ entry in the second column, and the procedure is iterated. Thus, the inverse takes the form
\[
\begin{bmatrix}
    1 & 0 & \cdots & 0 \\
    1 & M_{H',(1,1)}^{-1} & \cdots & M_{H',(1,n-1)}^{-1} \\
    \vdots & \vdots & \ddots & \vdots \\
    1 & M_{H',(n-1,1)}^{-1} & \cdots & M_{H',(n-1,n-1)}^{-1}
\end{bmatrix}
\]
In particular, notice that $M_H^{-1}(x_0e_1)=x_0\mathbf{1}$, where $\mathbf{1}$ is the vector whose entries are all equal to 1. With this established, we return to computing the density. Changing variables via $\mathbf{y}=M_H(\mathbf{x})-x_0e_1$ (we also write $\mathbf{y}=(y_1,\mathbf{y'})$ for convenience and let $t_j=|y_j|$) leads to the formula
\[
F(\phi)=\int_{\mathbf{y}}  \phi(\mathbf{t})\mu(y_1+x_0)(\prod_{j= 2}^n \mu)(M_{H}^{-1}(\mathbf{y})+x_0\mathbf{1})\,d\mathbf{y}.
\]
The form of the inverse $M_H^{-1}$ guarantees that a $x_0+y_1$ summand appears in each argument of a $\mu$ factor. Going to polar coordinates $y_i=t_i\omega_i$, $1\leq i\leq n$, we deduce immediately the desired form
\[
\nu_{\mathcal{T},{v_0}}^{x_0}(\mathbf{t})=A_{t_1}(\mu(\cdot) \nu_{H',v_1}^{(\cdot)}(t_2,t_3,\dots, t_n))(x_0).
\]
\end{proof}
Note that for the tree given in Example \ref{figure1}, by using 
$$\Phi^{x_0}(x_1,x_2,x_3,x_4)=(|x_1-x_0|, |x_2-x_0|,|x_3-x_2|,|x_4-x_0|)$$
we can explicitly compute
\[
\nu^x_{\mathcal{T},v_0}(t_1,t_2,t_3,t_4)=A_{t_1}(\mu)(x)A_{t_2}(\mu A_{t_3}(\mu))(x)A_{t_4}(\mu)(x),
\]
which is a product of 1-chain operators (coming from $v_1$ and $v_4$) and a 2-chain operator pinned at $v_0$, coming from the subtree spanned by $v_0,v_2,$ and $v_3$.

We are finally ready to state our main technical result, which will imply the pinned nonempty interior result for trees. Recall the definition of $(1-\epsilon)$-restriction of a measure in Subsection \ref{subsection: restrictmeasures}. In the proposition below we will use $A_{\mathbf{t}}^{(\mathcal{T}^j,v_0)}(\mu,\mu,\dots,\mu)(x)$ 
 as an alternative notation for $\nu_{\mathcal{T}^j,v_0}^{x}(\mathbf{t})$, which is supported in $\Delta_{\mathcal{T}^j,v_0}^{x}(E)$.

\begin{prop}\label{lem tree induction}
    Let $d\ge 2$ and fix a Borel set $E\subset \R^d$. Assume that the local smoothing estimate \eqref{partiallocalsmoothing} holds for some $p\in (2,\infty)$ and $\eta>0$, and that $\Hdim(E)>d-\eta$. For every $n\in \N$ and every $\epsilon>0$, there exists $\gamma>\frac{1}{p}$ and a sequence of Frostman probability measures $\{\mu_1,\ldots,\mu_n\}$ supported in $E$ satisfying $\mu_n$ is a $(1-\epsilon)$-restriction of $\mu_1$ and that $\mu_{i'}$ is a restriction of $\mu_i$ whenever $i'\ge i$, such that for every tree $\mathcal{T}^k$ with $k+1$ vertices with a distinguished vertex $v_0$, where $k\le n$, and such that for every $\tau>0$ and every $\tau$-restriction measure $\nu_n$ of $\mu_n$, we have that
    \[
    \Vert\partial_{t_1, \dots, t_k}^{\gamma} A_{t_1,\dots, t_k}^{(\mathcal{T}^k,v_0)}(\nu_n,\ldots,\nu_n)\Vert_{L^{p}(d\nu_n\,dt_1\dots dt_k)}\le C_{n,\epsilon,\tau,d}
    \]
    where $\partial_{t_1, \dots, t_k}^{\gamma}$ is a shortcut notation for $\partial_{t_1}^{\gamma}\partial_{t_2}^{\gamma}\dots \partial_{t_k}^{\gamma}$.
  
\end{prop}
\begin{proof}
    We will proceed by strong induction on $n$: the base case follows by \eqref{reduction1} and the remark at the end of the proof of Theorem \ref{Thm:pinnednonemptyassumingLS} concerning restriction measures, so we will assume that we know the result for all $k\le n-1$. Fix a tree $\mathcal{T}^n$ with $n+1$ vertices and distinguished vertex $v_0$. Let $\mu_{n-1}$ be a Frostman probability measure that works with the parameter $\tau/2$ for all trees with at most $n$ vertices and is a $(1-c_n\epsilon)$-restriction of the $\mu_1$ measure that exists from the base case, where $c_n$ is sufficiently small. For every tree $\mathcal{T}^r$ with at most $r\le n$ vertices and a distinguished vertex $v_0$, we consider the function
    \[
    h_{\mathcal{T}^r,v_0}(x)= \| \partial_{t_1,\dots,t_r }^{\gamma}A_{t_1,\dots,t_r}^{(\mathcal{T}^r,v_0)}(\mu_{n-1},\ldots,\mu_{n-1})(x)\|_{L^p(dt_1\dots dt_r)}. 
    \]

   Thus, since $\mathcal{T}^r$ has at most $n$ vertices, our induction hypothesis guarantees that $ h_{\mathcal{T}^r,v_0}\in L^p(\mu_{n-1})$. We do this for all such pairs of trees and pins $(\mathcal{T}^r,v_0)$ where $r\le n$. Since there are finitely many such pairs we still have
\[h_{r,v_0}(x):=\sum_{(\mathcal{T}^r,v_0)}h_{\mathcal{T}^r,v_0}(x)\in L^p(\mu_{n-1}).\]
     
     Hence, for every $\beta>0$, there exists a set $E_{\beta}\subset E$ such that $\mu_{n-1}(E_{\beta})>1-\beta$ and $|h_{r}(x)|<N_{\beta}$ for all $x\in E_{\beta}$ and some sufficiently large $N_{\beta}$. We can choose $0<\beta_n\ll\epsilon$ so that $E_{n}:=E_{\beta_n}$ is a $(1-\epsilon)$-restriction of $\mu_1$. We choose the restricted probability measure $(\mu_{n-1})_{E_{n}}$ to be $\mu_n$. Notice that by our choice, we will have that for all $(\mathcal{T}^r,v_0)$ with $r\leq n$,
    $$|h_{\mathcal{T}^r,v_0}(x)|\leq N_{\epsilon,n},\,\text{ for all }x\in \text{supp}(\mu_n).$$

     Now, we split into two cases. \\ \\
    \textbf{Case 1: $v_0$ is not a leaf.} In particular, deleting the vertex $v_0$ results in $m\ge 2$ new trees. We will consider the pinned trees $H_1,\cdots,H_m$, each with the pin $v_0$. Recalling the notation $\bm{t}=(t_1,t_2,\dots, t_n)$ and $\bm{t_l}=(t_{l,1}, t_{l,2}, \dots, t_{l,n_l})$. Our goal is to bound
    \begin{equation}
\Vert\partial_{\bm{t}}^{\gamma}\prod_{l=1}^m A^{(H_l,v_0)}_{\bm{t}_l}(\nu_n,\ldots,\nu_n)\Vert_{L^{p}(d\nu_n\,d\bm{t})}\leq C_{n,d,\varepsilon,\tau}.
    \end{equation}
    where $\nu_n$ is any $\tau$-restriction of $\mu_n$. By Hölder's inequality, for each $x\in \supp(\nu_n)\subset E_n$, we bound 
    \begin{align*}
    \Vert\partial_{\bm{t}}^{\gamma}&\prod_{l=1}^m A^{(H_l,v_0)}_{\bm{t}_l}(\nu_n,\ldots,\nu_n)(x)\Vert_{L^{p}(d\bm{t})}\\
\leq&\Vert\partial_{\bm{t_1}}^{\gamma} A^{(H_1,v_0)}_{\bm{t}_{1}}(\nu_n,\ldots,\nu_n)(x)\Vert_{L^{p}(d\bm{t}_{1})}\prod_{l=2}^m\Vert \partial^{\gamma}_{\bm{t}_l}A^{(H_l,v_0)}_{\bm{t}_l}(\nu_n,\ldots,\nu_n)(x)\Vert_{L^{\infty}(d\bm{t}_l)}.
    \end{align*}

    The first factor is bounded by a constant $C_{n-1,\epsilon,\tau,d}$ by the induction hypothesis, where we are using that a $\tau$-restriction of a $(1-\epsilon)$-restriction of $\mu_1$ is also a $\tau/2$-restriction of $\mu_1$. The second factor is bounded by $(N_{\epsilon,n})^{m-1}$, by construction of $\mu_n$, which completes this case. \\ \\
    \textbf{Case 2: $v_0$ is a leaf.} In this case, let $H'$ denote the subtree obtained by deleting $v_0$ from $\mathcal{T}^n$ and let $v_1$ denote the unique vertex such that $(v_0,v_1)$ is an edge in $\mathcal{T}^n$. We have the formula
    \[
    A_{\mathbf{t}}^{(\mathcal{T},v_0)}(\nu_n,\ldots,\nu_n)(x)=A_{t_1}(\nu_n A^{(H',v_1)}_{t_2,\ldots,t_n}(\nu_n,\ldots,\nu_n))(x),
    \]
    where once again, $\nu_n$ is a $\tau$-restriction of $\mu_n$.
    
    By duality, there is a function $g$ with $\Vert g\Vert_{L^{p'}(d\nu_n\,d\bm{t})}=1$, where we also use the notation $g_{\bm{t}}(y)=g(y,t_1,\dots, t_n)$, such that
    \[
    \Vert \partial_{\bm{t}}^{\gamma} A_{\bm{t}}^{(\mathcal{T},v_0)}(\nu_n,\ldots,\nu_n)\Vert_{L^{p}(d\nu_n\,d\bm{t})}=\left|\int \partial_{\bm{t}}^{\gamma} A_{\mathbf{t}}^{(\mathcal{T},v_0)}(\nu_n,\ldots,\nu_n)(y) g(y,\bm{t})\,d\nu_n(y)\,d\bm{t}\right|.
    \]
    We will use the notation $\mathbf{t}=(t_1,\mathbf{t}')$. We now take Littlewood-Paley expansions and estimate
    \begin{align*}
        & \left| \int \partial_{t_1}^{\gamma} A_{t_1} (\nu_n \partial^{\gamma}_{\bm{t}'}A^{(H',v_1)}_{\mathbf{t}'}(\nu_n,\ldots,\nu_n)(y)\,g_{\bm{t}}(y)\,d\nu_n(y)\,dt_1\,d\bm{t}'\right| \\
        \lesssim& \sum_{j=0}^{\infty}\left|\int_{[1,2]^{n-1}}\int_1^2 \int_{\R^d} \partial^{\gamma}_{t_1}A_{t_1}(P_j(\nu_n \partial^{\gamma}_{\bm{t}'}A_{\mathbf{t}'}^{(H',v_1)}(\nu_n,\ldots,\nu_n)))(y)g_{\bm{t}}(y)\,d\nu_n(y)\,dt_1\,d\bm{t}'\right| \\
        \lesssim & \sum_{j=0}^{\infty}\left|\int_{[1,2]^{n-1}}\int_1^2 \int_{\R^d} \partial^{\gamma}_{t_1}A_{t_1}(P_j(\nu_n \partial^{\gamma}_{\bm{t}'} A_{\mathbf{t}'}^{(H',v_1)}(\nu_n,\ldots,\nu_n)))(y) \sum_{i=-2}^2 (P_{j+i} [g_{\bm{t}}d\nu_n])dy\,dt_1\,d\bm{t}'\right| \\
        \lesssim & \sum_{j=0}^{\infty}\left(\Vert \partial^{\gamma}_{t_1}A_{t_1}(P_j(\nu_n \partial^{\gamma}_{\bm{t}'}A_{\mathbf{t}'}^{(H',v_1)}(\nu_n,\ldots,\nu_n)))\Vert_{L^{p}(dy\,dt_1\,d\bm{t}')} \sum_{i=-2}^2\Vert P_{j+i} (g_{\bm{t}}\nu_n)\Vert_{L^{p'}(dy\,dt_1\,d\bm{t}')}\right).
    \end{align*}
    Applying Lemma \ref{lem proj of product} with each $\bm{t}\in [1,2]^n$ fixed, let us bound the final factor via
    \[
\sum_{i=-2}^2 \Vert P_{j+i} (g_{\bm{t}}\nu_n)\Vert_{L^{p'}(dydt_1d\bm{t}')} \lesssim 2^{\frac{j}{p}(d-s_E+\varepsilon)}\|\|g_{\bm{t}}\|_{L^{p'}(d\nu_n)}\|_{L^p(d\bm{t})}=2^{\frac{j}{p}(d-s_E+\varepsilon)}.\]
    For the first factor, we use local smoothing and Lemma \ref{lem proj of product} to estimate for $\gamma=1/p+\varepsilon$ and for each $\bm{t}'\in [1,2]^{n-1}$,
    \begin{align*}
    \Vert \partial^{\gamma}_{t_1}A_{t_1}&(P_j(\nu_n \partial_{\bm{t}'}^{\gamma}A_{\mathbf{t}'}^{(H',v_1)}(\nu_n,\ldots,\nu_n)))\Vert_{L^{p}(dy\,dt_1)} \\
    &\lesssim 2^{j(-\eta(p,d)+2\varepsilon)}\Vert P_j(\nu_n \partial_{\bm{t}'}^{\gamma}A_{\mathbf{t}'}^{(H',v_1)}(\nu_n,\ldots,\nu_n))\Vert_{L^{p}(dy)} \\
    &\lesssim 2^{j(-\eta(p,d)+2\varepsilon)}2^{\frac{j}{p'}(d-s_E+\varepsilon)}\Vert \partial_{\bm{t}'}^{\gamma}A_{\mathbf{t}'}^{(H',v_1)}(\nu_n,\ldots,\nu_n)\Vert_{L^{p}(d\nu_n)}.
    \end{align*}

    Hence,
    \begin{align*}
    \Vert \partial^{\gamma}_{t_1}A_{t_1}&(P_j(\nu_n \partial_{\bm{t}'}^{\gamma}A_{\mathbf{t}'}^{(H',v_1)}(\nu_n,\ldots,\nu_n)))\Vert_{L^{p}(dy\,d\bm{t})} \\
    &\lesssim 2^{j(-\eta(p,d)+2\varepsilon)}2^{\frac{j}{p'}(d-s_E+\varepsilon)}\Vert \partial_{\bm{t}'}^{\gamma}A_{\mathbf{t}'}^{(H',v_1)}(\nu_n,\ldots,\nu_n)\Vert_{L^{p}(d\nu_nd\bm{t})}.
    \end{align*}
    
    The induction hypothesis implies that
    \[
 \Vert \partial_{\bm{t}'}^{\gamma}A_{\mathbf{t}'}^{(H',v_1)}(\nu_n,\ldots,\nu_n)\Vert_{L^{p}(d\nu_nd\bm{t})}\leq C_{n-1,\epsilon,\tau,d}.
    \]
    Substituting all the bounds together leads to
    \[
     \Vert \partial_{\bm{t}}^{\gamma} A_{\bm{t}}^{(\mathcal{T},v_0)}(\nu_n,\ldots,\nu_n)\Vert_{L^{p}(d\nu_n\,d\bm{t})}\lesssim C_{n-1,\epsilon,\tau,d}\sum_{j=0}^{\infty}2^{j(-\eta(p,d)+2\varepsilon)}2^{j(d-s_E+\varepsilon)}.
    \]
    Thus, we deduce that if $\varepsilon$ is sufficiently small (depending on $s_E$), then the sum converges and 
    \[ \Vert \partial_{\bm{t}}^{\gamma} A_{\bm{t}}^{(\mathcal{T},v_0)}(\nu_n,\ldots,\nu_n)\Vert_{L^{p}(d\nu_n\,d\bm{t})}\leq C_{n,\epsilon, \tau,d}\]

\end{proof}
Combining Proposition \ref{lem tree induction} with the Sobolev embedding theorem implies that measure $\nu^{x_0}_{\mathcal{T},v_0}$ has continuous density. We deduce Theorem \ref{Thm:pinnednonemptytreeinhigherdim} as an immediate corollary.

\bibliographystyle{alpha}
\bibliography{sources}

@article {IL19,
    AUTHOR = {Iosevich, Alex and Liu, Bochen},
     TITLE = {Pinned distance problem, slicing measures, and local smoothing
              estimates},
   JOURNAL = {Trans. Amer. Math. Soc.},
  FJOURNAL = {Transactions of the American Mathematical Society},
    VOLUME = {371},
      YEAR = {2019},
    NUMBER = {6},
     PAGES = {4459--4474},
      ISSN = {0002-9947,1088-6850},
   MRCLASS = {28A75 (42B20)},
  MRNUMBER = {3917228},
MRREVIEWER = {Xiumin\ Du},
       DOI = {10.1090/tran/7693},
       URL = {https://doi.org/10.1090/tran/7693},
}

@article {GWZ20,
    AUTHOR = {Guth, Larry and Wang, Hong and Zhang, Ruixiang},
     TITLE = {A sharp square function estimate for the cone in {$\Bbb
              {R}^3$}},
   JOURNAL = {Ann. of Math. (2)},
  FJOURNAL = {Annals of Mathematics. Second Series},
    VOLUME = {192},
      YEAR = {2020},
    NUMBER = {2},
     PAGES = {551--581},
      ISSN = {0003-486X,1939-8980},
   MRCLASS = {42B15 (35L05)},
  MRNUMBER = {4151084},
MRREVIEWER = {Huoxiong\ Wu},
       DOI = {10.4007/annals.2020.192.2.6},
       URL = {https://doi.org/10.4007/annals.2020.192.2.6},
}

@article {GLMX23,
    AUTHOR = {Gao, Chuanwei and Liu, Bochen and Miao, Changxing and Xi,
              Yakun},
     TITLE = {Square function estimates and local smoothing for {F}ourier
              integral operators},
   JOURNAL = {Proc. Lond. Math. Soc. (3)},
  FJOURNAL = {Proceedings of the London Mathematical Society. Third Series},
    VOLUME = {126},
      YEAR = {2023},
    NUMBER = {6},
     PAGES = {1923--1960},
      ISSN = {0024-6115,1460-244X},
   MRCLASS = {42B37 (35L05 35S30)},
  MRNUMBER = {4601326},
MRREVIEWER = {Yunyun\ Hu},
       DOI = {10.1112/plms.12521},
       URL = {https://doi.org/10.1112/plms.12521},
}

@article{GanWu,
      title={On local smoothing estimates for wave equations}, 
      author={Shengwen Gan and Shukun Wu},
    journal={Preprint arXiv:2502.05973},
      year={2025},
      eprint={2502.05973},
      archivePrefix={arXiv},
      primaryClass={math.AP},
      url={https://arxiv.org/abs/2502.05973}, 
}

@article {PS00,
    AUTHOR = {Peres, Yuval and Schlag, Wilhelm},
     TITLE = {Smoothness of projections, {B}ernoulli convolutions, and the
              dimension of exceptions},
   JOURNAL = {Duke Math. J.},
  FJOURNAL = {Duke Mathematical Journal},
    VOLUME = {102},
      YEAR = {2000},
    NUMBER = {2},
     PAGES = {193--251},
      ISSN = {0012-7094,1547-7398},
   MRCLASS = {42B25 (28A78)},
  MRNUMBER = {1749437},
MRREVIEWER = {Esa\ J\"arvenp\"a\"a},
       DOI = {10.1215/S0012-7094-00-10222-0},
       URL = {https://doi.org/10.1215/S0012-7094-00-10222-0},
}

@misc{BIO23,
      title={A singular variant of the {F}alconer distance problem}, 
      author={Tainara Borges and Alex Iosevich and Yumeng Ou},
      journal={Preprint arXiv:2306.05247},
      year={2023},
      eprint={2306.05247},
      archivePrefix={arXiv},
      primaryClass={math.CA},
      url={https://arxiv.org/abs/2306.05247}, 
}

@article{IKSTU,
    AUTHOR = {Iosevich, A. and Krause, B. and Sawyer, E. and Taylor, K. and
              Uriarte-Tuero, I.},
     TITLE = {Maximal operators: scales, curvature and the fractal
              dimension},
   JOURNAL = {Anal. Math.},
  FJOURNAL = {Analysis Mathematica},
    VOLUME = {45},
      YEAR = {2019},
    NUMBER = {1},
     PAGES = {63--86},
      ISSN = {0133-3852,1588-273X},
   MRCLASS = {42B37 (11K55 28A75 35S30)},
  MRNUMBER = {3916130},
MRREVIEWER = {Xiumin\ Du},
       DOI = {10.1007/s10476-018-0307-9},
       URL = {https://doi.org/10.1007/s10476-018-0307-9},
}

@book {Mattilabook2015,
    AUTHOR = {Mattila, Pertti},
     TITLE = {Fourier analysis and {H}ausdorff dimension},
    SERIES = {Cambridge Studies in Advanced Mathematics},
    VOLUME = {150},
 PUBLISHER = {Cambridge University Press, Cambridge},
      YEAR = {2015},
     PAGES = {xiv+440},
      ISBN = {978-1-107-10735-9},
   MRCLASS = {28-02 (28A15 28A78 28A80 42B10 60J65)},
  MRNUMBER = {3617376},
MRREVIEWER = {Benjamin\ Steinhurst},
       DOI = {10.1017/CBO9781316227619},
       URL = {https://doi.org/10.1017/CBO9781316227619},
}

@article {MattilaSjolin,
    AUTHOR = {Mattila, Pertti and Sj\"olin, Per},
     TITLE = {Regularity of distance measures and sets},
   JOURNAL = {Math. Nachr.},
  FJOURNAL = {Mathematische Nachrichten},
    VOLUME = {204},
      YEAR = {1999},
     PAGES = {157--162},
      ISSN = {0025-584X,1522-2616},
   MRCLASS = {42B10 (28A75)},
  MRNUMBER = {1705134},
MRREVIEWER = {Joan\ Orobitg},
       DOI = {10.1002/mana.3212040110},
       URL = {https://doi.org/10.1002/mana.3212040110},
}

@article {OuTaylor,
    AUTHOR = {Ou, Yumeng and Taylor, Krystal},
     TITLE = {Finite point configurations and the regular value theorem in a
              fractal setting},
   JOURNAL = {Indiana Univ. Math. J.},
  FJOURNAL = {Indiana University Mathematics Journal},
    VOLUME = {71},
      YEAR = {2022},
    NUMBER = {4},
     PAGES = {1707--1761},
      ISSN = {0022-2518,1943-5258},
   MRCLASS = {42B10 (28A78 52C10)},
  MRNUMBER = {4481098},
MRREVIEWER = {Rami\ Ayoush},
}

@article {Falconer85,
    AUTHOR = {Falconer, K. J.},
     TITLE = {On the {H}ausdorff dimensions of distance sets},
   JOURNAL = {Mathematika},
  FJOURNAL = {Mathematika. A Journal of Pure and Applied Mathematics},
    VOLUME = {32},
      YEAR = {1985},
    NUMBER = {2},
     PAGES = {206--212},
      ISSN = {0025-5793},
   MRCLASS = {28A75 (28A05)},
  MRNUMBER = {834490},
MRREVIEWER = {S.\ J.\ Taylor},
       DOI = {10.1112/S0025579300010998},
       URL = {https://doi.org/10.1112/S0025579300010998},
}

@article {Bourgain94,
    AUTHOR = {Bourgain, Jean},
     TITLE = {Hausdorff dimension and distance sets},
   JOURNAL = {Israel J. Math.},
  FJOURNAL = {Israel Journal of Mathematics},
    VOLUME = {87},
      YEAR = {1994},
    NUMBER = {1-3},
     PAGES = {193--201},
      ISSN = {0021-2172,1565-8511},
   MRCLASS = {28A78},
  MRNUMBER = {1286826},
MRREVIEWER = {K.\ E.\ Hirst},
       DOI = {10.1007/BF02772994},
       URL = {https://doi.org/10.1007/BF02772994},
}

@article {Wolff99,
    AUTHOR = {Wolff, Thomas},
     TITLE = {Decay of circular means of {F}ourier transforms of measures},
   JOURNAL = {Internat. Math. Res. Notices},
  FJOURNAL = {International Mathematics Research Notices},
      YEAR = {1999},
    NUMBER = {10},
     PAGES = {547--567},
      ISSN = {1073-7928,1687-0247},
   MRCLASS = {42B10 (42B25)},
  MRNUMBER = {1692851},
MRREVIEWER = {Steen\ Pedersen},
       DOI = {10.1155/S1073792899000288},
       URL = {https://doi.org/10.1155/S1073792899000288},
}

@article {Erdogan05,
    AUTHOR = {Erdo\u{g}an, M. Burak},
     TITLE = {A bilinear {F}ourier extension theorem and applications to the
              distance set problem},
   JOURNAL = {Int. Math. Res. Not.},
  FJOURNAL = {International Mathematics Research Notices},
      YEAR = {2005},
    NUMBER = {23},
     PAGES = {1411--1425},
      ISSN = {1073-7928,1687-0247},
   MRCLASS = {42B20 (28A75)},
  MRNUMBER = {2152236},
MRREVIEWER = {Herv\'e\ Pajot},
       DOI = {10.1155/IMRN.2005.1411},
       URL = {https://doi.org/10.1155/IMRN.2005.1411},
}

@article {Liu19,
    AUTHOR = {Liu, Bochen},
     TITLE = {An {$L^2$}-identity and pinned distance problem},
   JOURNAL = {Geom. Funct. Anal.},
  FJOURNAL = {Geometric and Functional Analysis},
    VOLUME = {29},
      YEAR = {2019},
    NUMBER = {1},
     PAGES = {283--294},
      ISSN = {1016-443X,1420-8970},
   MRCLASS = {28A75 (28A78 42B10)},
  MRNUMBER = {3925111},
MRREVIEWER = {Lars\ Olsen},
       DOI = {10.1007/s00039-019-00482-8},
       URL = {https://doi.org/10.1007/s00039-019-00482-8},
}

@article {GIOW20,
    AUTHOR = {Guth, Larry and Iosevich, Alex and Ou, Yumeng and Wang, Hong},
     TITLE = {On {F}alconer's distance set problem in the plane},
   JOURNAL = {Invent. Math.},
  FJOURNAL = {Inventiones Mathematicae},
    VOLUME = {219},
      YEAR = {2020},
    NUMBER = {3},
     PAGES = {779--830},
      ISSN = {0020-9910,1432-1297},
   MRCLASS = {42B20 (28A80)},
  MRNUMBER = {4055179},
MRREVIEWER = {Jonathan\ MacDonald\ Fraser},
       DOI = {10.1007/s00222-019-00917-x},
       URL = {https://doi.org/10.1007/s00222-019-00917-x},
}

@article {DGOWWZ21,
    AUTHOR = {Du, Xiumin and Guth, Larry and Ou, Yumeng and Wang, Hong and
              Wilson, Bobby and Zhang, Ruixiang},
     TITLE = {Weighted restriction estimates and application to {F}alconer
              distance set problem},
   JOURNAL = {Amer. J. Math.},
  FJOURNAL = {American Journal of Mathematics},
    VOLUME = {143},
      YEAR = {2021},
    NUMBER = {1},
     PAGES = {175--211},
      ISSN = {0002-9327,1080-6377},
   MRCLASS = {28A80 (28A75 42B10 42B20)},
  MRNUMBER = {4201782},
MRREVIEWER = {Zolt\'an\ L.\ Buczolich},
       DOI = {10.1353/ajm.2021.0005},
       URL = {https://doi.org/10.1353/ajm.2021.0005},
}

@article {DZ19,
    AUTHOR = {Du, Xiumin and Zhang, Ruixiang},
     TITLE = {Sharp {$L^2$} estimates of the {S}chr\"odinger maximal
              function in higher dimensions},
   JOURNAL = {Ann. of Math. (2)},
  FJOURNAL = {Annals of Mathematics. Second Series},
    VOLUME = {189},
      YEAR = {2019},
    NUMBER = {3},
     PAGES = {837--861},
      ISSN = {0003-486X,1939-8980},
   MRCLASS = {42B20 (42B37)},
  MRNUMBER = {3961084},
MRREVIEWER = {Dong\ Dong},
       DOI = {10.4007/annals.2019.189.3.4},
       URL = {https://doi.org/10.4007/annals.2019.189.3.4},
}

@article {DIOWZ21,
    AUTHOR = {Du, Xiumin and Iosevich, Alex and Ou, Yumeng and Wang, Hong
              and Zhang, Ruixiang},
     TITLE = {An improved result for {F}alconer's distance set problem in
              even dimensions},
   JOURNAL = {Math. Ann.},
  FJOURNAL = {Mathematische Annalen},
    VOLUME = {380},
      YEAR = {2021},
    NUMBER = {3-4},
     PAGES = {1215--1231},
      ISSN = {0025-5831,1432-1807},
   MRCLASS = {28A80},
  MRNUMBER = {4297185},
MRREVIEWER = {Chol-Hui\ Yun},
       DOI = {10.1007/s00208-021-02170-1},
       URL = {https://doi.org/10.1007/s00208-021-02170-1},
}

@misc{DORZ23,
      title={Weighted refined decoupling estimates and application to {F}alconer distance set problem}, 
      author={Du, Xiumin and Ou, Yumeng and Ren, Kevin and Zhang, Ruixiang},
      year={2023},
      eprint={2309.04501},
      archivePrefix={arXiv},
      primaryClass={math.CA},
      url={https://arxiv.org/abs/2309.04501}, 
}

@article {Steinhaus20,
    AUTHOR = {Steinhaus, Hugo},
     TITLE = {Sur les distances des points dans les ensembles de mesure positive},
   JOURNAL = {Fund. Math.},
  FJOURNAL = {Fundamenta Mathematicae},
    VOLUME = {1},
      YEAR = {1920},
    NUMBER = {1},
     PAGES = {93--104},
      ISSN = {0016-2736},
}

@article {SW25,
    AUTHOR = {Shmerkin, Pablo and Wang, Hong},
     TITLE = {On the distance sets spanned by sets of dimension $\frac{d}{2}$ in $\mathbb{R}^d$},
   JOURNAL = {Geom. Funct. Anal.},
  FJOURNAL = {Geometric and Functional Analysis },
      YEAR = {2025},
       DOI = {10.1007/s00039-024-00696-5},
}

@article {IMT12,
    AUTHOR = {Iosevich, Alex and Mourgoglou, Mihalis and Taylor, Krystal},
     TITLE = {On the {M}attila-{S}j\"olin theorem for distance sets},
   JOURNAL = {Ann. Acad. Sci. Fenn. Math.},
  FJOURNAL = {Annales Academi\ae\ Scientiarum Fennic\ae. Mathematica},
    VOLUME = {37},
      YEAR = {2012},
    NUMBER = {2},
     PAGES = {557--562},
      ISSN = {1239-629X,1798-2383},
   MRCLASS = {28A75 (42B20 52C10)},
  MRNUMBER = {2987085},
MRREVIEWER = {Alain\ Rivi\`ere},
       DOI = {10.5186/aasfm.2012.3732},
       URL = {https://doi.org/10.5186/aasfm.2012.3732},
}

@article {BIT16,
    AUTHOR = {Bennett, Michael and Iosevich, Alexander and Taylor, Krystal},
     TITLE = {Finite chains inside thin subsets of {$\Bbb{R}^d$}},
   JOURNAL = {Anal. PDE},
  FJOURNAL = {Analysis \& PDE},
    VOLUME = {9},
      YEAR = {2016},
    NUMBER = {3},
     PAGES = {597--614},
      ISSN = {2157-5045,1948-206X},
   MRCLASS = {28A75 (42B10)},
  MRNUMBER = {3518531},
MRREVIEWER = {Gareth\ Speight},
       DOI = {10.2140/apde.2016.9.597},
       URL = {https://doi.org/10.2140/apde.2016.9.597},
}

@incollection {IT19,
    AUTHOR = {Iosevich, A. and Taylor, K.},
     TITLE = {Finite trees inside thin subsets of {$\Bbb R^d$}},
 BOOKTITLE = {Modern methods in operator theory and harmonic analysis},
    SERIES = {Springer Proc. Math. Stat.},
    VOLUME = {291},
     PAGES = {51--56},
 PUBLISHER = {Springer, Cham},
      YEAR = {2019},
      ISBN = {978-3-030-26748-3; 978-3-030-26747-6},
   MRCLASS = {42B10 (05C05 28A80)},
  MRNUMBER = {4008977},
MRREVIEWER = {Peter\ R.\ Massopust},
       DOI = {10.1007/978-3-030-26748-3\_3},
       URL = {https://doi.org/10.1007/978-3-030-26748-3_3},
}

@article {GIT21,
    AUTHOR = {Greenleaf, Allan and Iosevich, Alex and Taylor, Krystal},
     TITLE = {Configuration sets with nonempty interior},
   JOURNAL = {J. Geom. Anal.},
  FJOURNAL = {Journal of Geometric Analysis},
    VOLUME = {31},
      YEAR = {2021},
    NUMBER = {7},
     PAGES = {6662--6680},
      ISSN = {1050-6926,1559-002X},
   MRCLASS = {28A80 (35S30 44A12)},
  MRNUMBER = {4289240},
       DOI = {10.1007/s12220-019-00288-y},
       URL = {https://doi.org/10.1007/s12220-019-00288-y},
}

@article {GIT22,
    AUTHOR = {Greenleaf, Allan and Iosevich, Alex and Taylor, Krystal},
     TITLE = {On {$k$}-point configuration sets with nonempty interior},
   JOURNAL = {Mathematika},
  FJOURNAL = {Mathematika. A Journal of Pure and Applied Mathematics},
    VOLUME = {68},
      YEAR = {2022},
    NUMBER = {1},
     PAGES = {163--190},
      ISSN = {0025-5793,2041-7942},
   MRCLASS = {28A75 (28A80 52C10 58J40)},
  MRNUMBER = {4405974},
MRREVIEWER = {Xiumin\ Du},
       DOI = {10.1112/mtk.12114},
       URL = {https://doi.org/10.1112/mtk.12114},
}

@article {GIT24,
    AUTHOR = {Greenleaf, Allan and Iosevich, Alex and Taylor, Krystal},
     TITLE = {Nonempty interior of configuration sets via microlocal
              partition optimization},
   JOURNAL = {Math. Z.},
  FJOURNAL = {Mathematische Zeitschrift},
    VOLUME = {306},
      YEAR = {2024},
    NUMBER = {4},
     PAGES = {Paper No. 66, 20},
      ISSN = {0025-5874,1432-1823},
   MRCLASS = {28A75 (28A80 52C10 58J40)},
  MRNUMBER = {4716767},
MRREVIEWER = {Stefan\ Steinerberger},
       DOI = {10.1007/s00209-024-03466-z},
       URL = {https://doi.org/10.1007/s00209-024-03466-z},
}

@article {GIT24preprint,
    AUTHOR = {Greenleaf, Allan and Iosevich, Alex and Taylor, Krystal},
     TITLE = {Realizing trees of configurations in thin sets},
   JOURNAL = {Pacific J. Math.},
  FJOURNAL = {Pacific Journal of Mathematics},
    VOLUME = {335},
      YEAR = {2025},
    NUMBER = {2},
     PAGES = {355--372},
      ISSN = {0030-8730,1945-5844},
   MRCLASS = {28A75 (42B35)},
  MRNUMBER = {4904870},
MRREVIEWER = {Bochen\ Liu},
       DOI = {10.2140/pjm.2025.335.355},
       URL = {https://doi.org/10.2140/pjm.2025.335.355},
}

@article {BORSS22,
    AUTHOR = {Beltran, David and Oberlin, Richard and Roncal, Luz and
              Seeger, Andreas and Stovall, Betsy},
     TITLE = {Variation bounds for spherical averages},
   JOURNAL = {Math. Ann.},
  FJOURNAL = {Mathematische Annalen},
    VOLUME = {382},
      YEAR = {2022},
    NUMBER = {1-2},
     PAGES = {459--512},
      ISSN = {0025-5831,1432-1807},
   MRCLASS = {42B15 (42B25)},
  MRNUMBER = {4377310},
MRREVIEWER = {Sundaram\ Thangavelu},
       DOI = {10.1007/s00208-021-02218-2},
       URL = {https://doi.org/10.1007/s00208-021-02218-2},
}

@article {Miyachi80,
    AUTHOR = {Miyachi, Akihiko},
     TITLE = {On some estimates for the wave equation in {$L\sp{p}$}\ and
              {$H\sp{p}$}},
   JOURNAL = {J. Fac. Sci. Univ. Tokyo Sect. IA Math.},
  FJOURNAL = {Journal of the Faculty of Science. University of Tokyo.
              Section IA. Mathematics},
    VOLUME = {27},
      YEAR = {1980},
    NUMBER = {2},
     PAGES = {331--354},
      ISSN = {0040-8980},
   MRCLASS = {35L05 (47G05)},
  MRNUMBER = {586454},
MRREVIEWER = {Author's review},
}

@article {Peral80,
    AUTHOR = {Peral, Juan C.},
     TITLE = {{$L\sp{p}$}\ estimates for the wave equation},
   JOURNAL = {J. Functional Analysis},
  FJOURNAL = {Journal of Functional Analysis},
    VOLUME = {36},
      YEAR = {1980},
    NUMBER = {1},
     PAGES = {114--145},
      ISSN = {0022-1236},
   MRCLASS = {35L05 (35B30)},
  MRNUMBER = {568979},
MRREVIEWER = {R.\ Glassey},
       DOI = {10.1016/0022-1236(80)90110-X},
       URL = {https://doi.org/10.1016/0022-1236(80)90110-X},
}

@article {Sogge91,
    AUTHOR = {Sogge, Christopher D.},
     TITLE = {Propagation of singularities and maximal functions in the
              plane},
   JOURNAL = {Invent. Math.},
  FJOURNAL = {Inventiones Mathematicae},
    VOLUME = {104},
      YEAR = {1991},
    NUMBER = {2},
     PAGES = {349--376},
      ISSN = {0020-9910,1432-1297},
   MRCLASS = {58G17 (35A20 42B25 58G15)},
  MRNUMBER = {1098614},
MRREVIEWER = {Min\ You\ Qi},
       DOI = {10.1007/BF01245080},
       URL = {https://doi.org/10.1007/BF01245080},
}

@incollection {EHI13,
    AUTHOR = {Erdo\u{g}an, Burak and Hart, Derrick and Iosevich, Alex},
     TITLE = {Multiparameter projection theorems with applications to
              sums-products and finite point configurations in the
              {E}uclidean setting},
 BOOKTITLE = {Recent advances in harmonic analysis and applications},
    SERIES = {Springer Proc. Math. Stat.},
    VOLUME = {25},
     PAGES = {93--103},
 PUBLISHER = {Springer, New York},
      YEAR = {2013},
      ISBN = {978-1-4614-4565-4; 978-1-4614-4564-7},
   MRCLASS = {28A80 (42B08)},
  MRNUMBER = {3066881},
MRREVIEWER = {Li-Feng\ Xi},
       DOI = {10.1007/978-1-4614-4565-4\_11},
       URL = {https://doi.org/10.1007/978-1-4614-4565-4_11},
}

@article {GI12,
    AUTHOR = {Greenleaf, Allan and Iosevich, Alex},
     TITLE = {On triangles determined by subsets of the {E}uclidean plane,
              the associated bilinear operators and applications to discrete
              geometry},
   JOURNAL = {Anal. PDE},
  FJOURNAL = {Analysis \& PDE},
    VOLUME = {5},
      YEAR = {2012},
    NUMBER = {2},
     PAGES = {397--409},
      ISSN = {2157-5045,1948-206X},
   MRCLASS = {42B15 (52C10)},
  MRNUMBER = {2970712},
MRREVIEWER = {Andreas\ Seeger},
       DOI = {10.2140/apde.2012.5.397},
       URL = {https://doi.org/10.2140/apde.2012.5.397},
}

@article {GILP15,
    AUTHOR = {Greenleaf, Allan and Iosevich, Alex and Liu, Bochen and
              Palsson, Eyvindur},
     TITLE = {A group-theoretic viewpoint on {E}rd\"os-{F}alconer problems
              and the {M}attila integral},
   JOURNAL = {Rev. Mat. Iberoam.},
  FJOURNAL = {Revista Matem\'atica Iberoamericana},
    VOLUME = {31},
      YEAR = {2015},
    NUMBER = {3},
     PAGES = {799--810},
      ISSN = {0213-2230,2235-0616},
   MRCLASS = {42B20 (52C10)},
  MRNUMBER = {3420476},
MRREVIEWER = {Tuomas\ P.\ Hyt\"onen},
       DOI = {10.4171/RMI/854},
       URL = {https://doi.org/10.4171/RMI/854},
}

@article {PR23,
    AUTHOR = {Palsson, Eyvindur Ari and Romero Acosta, Francisco},
     TITLE = {A {M}attila-{S}j\"olin theorem for triangles},
   JOURNAL = {J. Funct. Anal.},
  FJOURNAL = {Journal of Functional Analysis},
    VOLUME = {284},
      YEAR = {2023},
    NUMBER = {6},
     PAGES = {Paper No. 109814, 20},
      ISSN = {0022-1236,1096-0783},
   MRCLASS = {28A78 (42B20 52C10)},
  MRNUMBER = {4530888},
MRREVIEWER = {Vladimir\ Eiderman},
       DOI = {10.1016/j.jfa.2022.109814},
       URL = {https://doi.org/10.1016/j.jfa.2022.109814},
}

@article {PR25,
    AUTHOR = {Palsson, Eyvindur Ari and Romero Acosta, Francisco},
     TITLE = {A {M}attila--{S}j\"olin theorem for simplices in low
              dimensions},
   JOURNAL = {Math. Ann.},
  FJOURNAL = {Mathematische Annalen},
    VOLUME = {391},
      YEAR = {2025},
    NUMBER = {1},
     PAGES = {1123--1146},
      ISSN = {0025-5831,1432-1807},
   MRCLASS = {28A75 (28A78 42B20 52A20)},
  MRNUMBER = {4846807},
       DOI = {10.1007/s00208-024-02948-z},
       URL = {https://doi.org/10.1007/s00208-024-02948-z},
}

@article {IPPS22preprint,
    AUTHOR = {Iosevich, Alex and Pham, Minh-Quy and Pham, Thang and Shen,
              Chun-Yen},
     TITLE = {Pinned simplices and connections to product of sets on
              paraboloids},
   JOURNAL = {Indiana Univ. Math. J.},
  FJOURNAL = {Indiana University Mathematics Journal},
    VOLUME = {74},
      YEAR = {2025},
    NUMBER = {3},
     PAGES = {647--668},
      ISSN = {0022-2518,1943-5258},
   MRCLASS = {52C10 (28A80 42B10)},
  MRNUMBER = {4946877},
}

@book {Triebel95,
    AUTHOR = {Triebel, Hans},
     TITLE = {Interpolation theory, function spaces, differential operators},
   EDITION = {Second},
 PUBLISHER = {Johann Ambrosius Barth, Heidelberg},
      YEAR = {1995},
     PAGES = {532},
      ISBN = {3-335-00420-5},
   MRCLASS = {46-02 (35Jxx 46E10 46E35 46M35)},
  MRNUMBER = {1328645},
}

\end{document}